\documentclass[12pt,oneside,american]{amsart}
\usepackage[T1]{fontenc}
\usepackage[latin9]{inputenc}
\usepackage{geometry}
\usepackage{babel}
\usepackage{mathrsfs}
\usepackage{mathtools}
\usepackage{amstext}
\usepackage{amsthm}
\usepackage{amssymb}
\usepackage{bbm}
\usepackage[svgnames]{xcolor}
\usepackage[bookmarksnumbered=true]{hyperref} 
\hypersetup{
	colorlinks = true,
	linkcolor = Blue,
	anchorcolor = blue,
	citecolor = Green,
	filecolor = blue,
	urlcolor = FireBrick
}

\makeatletter
\theoremstyle{plain}
\newtheorem{thm}{\protect\theoremname}[section]
\theoremstyle{remark}
\newtheorem{rem}[thm]{\protect\remarkname}
\theoremstyle{plain}
\newtheorem{lem}[thm]{\protect\lemmaname}
\theoremstyle{definition}
\newtheorem{defn}[thm]{\protect\definitionname}
\theoremstyle{plain}

\makeatother

\providecommand{\corollaryname}{Corollary}
\providecommand{\definitionname}{Definition}
\providecommand{\lemmaname}{Lemma}
\providecommand{\remarkname}{Remark}
\providecommand{\theoremname}{Theorem}

\def\R{{\mathbb R}}

\def\1/2{\frac{1}{2}}

\def\Isom{\mbox{\rm Isom}}

\def\vep{\varepsilon}

\begin{document}

\title[The Lewy-Stampacchia inequality]{The Lewy-Stampacchia Inequality for the Fractional Laplacian and Its Application to Anomalous Unidirectional Diffusion Equations}

\author{Pu-Zhao Kow}

\address{Department of Mathematics, National Taiwan University, Taipei 106.}

\email{d07221005@ntu.edu.tw}

\author{Masato Kimura}

\address{Faculty of Mathematics and Physics, Kanazawa University, Kakuma, Kanazawa 920-1192.}

\email{mkimura@se.kanazawa-u.ac.jp}

\begin{abstract}
In this paper, we consider a Lewy-Stampacchia-type inequality for the fractional Laplacian on a bounded domain in Euclidean space. Using this inequality, we can show the well-posedness of fractional-type anomalous unidirectional diffusion equations. This study is an extension of the work by Akagi-Kimura (2019) for the standard Laplacian. However, there exist several difficulties due to the nonlocal feature of the fractional Laplacian. We overcome those difficulties employing the Caffarelli-Silvestre extension of the fractional Laplacian.
\end{abstract}


\maketitle

\section{Introduction}

\subsection{Overview}

Let $\Omega$ be a bounded Lipschitz domain in $\mathbb{R}^{n}$ with $n\in\mathbb{N}$, and let $s\in(0,1)$. The main aim of this paper is to prove a Lewy-Stampacchia-type inequality for the spectral fractional Laplacian $(-\Delta)^{s}$. Using this inequality, we also prove the well-posedness of the following unidirectional evolution equation of the fractional-diffusion type: 
\begin{align*}
\partial_{t}u & =[-(-\Delta)^{s}u+f]_{+}\quad\text{in }\Omega\times(0,T],\\
u & =0\quad\text{on }\partial\Omega\times(0,T] \quad \text{if }s \in (1/2,1),\\
u|_{t=0} & =u_{0}\quad\text{in }\Omega,
\end{align*}
where $\partial_{t}=\frac{\partial}{\partial t}$, $f=f(x,t)$ and $u_{0}=u_{0}(x)$ are given functions.

\subsection{Fractional Laplacian and Sobolev spaces}

Before we state our main results, we first introduce some notations and the definition of $(-\Delta)^{s}$ for $s\in(0,1)$ as in Caffarelli and Stinga \cite{CS16}. First, we recall some Sobolev spaces introduced in \cite[Remark 2.1]{CS16} and \cite[Section 2]{NOS15}. For $s \in (0,1)$ and $s \neq 1/2$, let $H_{0}^{s}(\Omega)$ be the fractional Sobolev space given by the completion of $\mathcal{C}_{c}^{\infty}(\Omega)$ under the norm
\begin{equation}
\|\bullet\|_{H_{0}^{s}(\Omega)}^{2}:=\|\bullet\|_{L^{2}(\Omega)}^{2}+[\bullet]_{H_{0}^{s}(\Omega)}^{2},\label{eq:sobolev-norm}
\end{equation}
where the seminorm $[\bullet]_{H_{0}^{s}(\Omega)}$ is given by 
\[
[v]_{H_{0}^{s}(\Omega)}^{2}:=\int_{\Omega}\int_{\Omega}\frac{(v(x)-v(z))^{2}}{|x-z|^{n+2s}}\,dx\,dz.
\]
For $s=1/2$, we write $H_{0}^{1/2}(\Omega)=H_{00}^{1/2}(\Omega)$, which is the Lions-Magenes space; that is, we additionally assume (this assumption is essential, see Remark~\ref{rem:equivalent-def-frac-space} for further explanation)
\begin{equation}
\int_{\Omega}\frac{u(x)^{2}}{{\rm dist}(x,\partial\Omega)}\,dx<\infty \label{eq:Lions-Magnes-condition}
\end{equation}
for the case $s=1/2$, and the norm is also given in \eqref{eq:sobolev-norm} (see \cite[equation (2.9)]{NOS15} and \cite[Theorem 11.7]{LM72} for further details). Because $\Omega$ has a Lipschitz boundary, by \cite[Theorem 2.9 and Theorem 2.12]{Mi11}, then  
\begin{subequations}
\begin{align}
H_{0}^{s}(\Omega) & =\big\{ v\in L^{2}(\Omega):~[v]_{H_{0}^{s}(\Omega)}<\infty\big\}\quad\text{for }0<s<\frac{1}{2},\label{eq:tracea}\\
H_{0}^{\frac{1}{2}}(\Omega) & =\big\{ v\in L^{2}(\Omega):~[v]_{H_{0}^{\frac{1}{2}}(\Omega)}<\infty,
\, v\text{ satisfies \eqref{eq:Lions-Magnes-condition}}\big\},\label{eq:traceb}\\
H_{0}^{s}(\Omega) & \subsetneq\big\{ v\in L^{2}(\Omega):~[v]_{H_{0}^{s}(\Omega)}<\infty\big\}\quad\text{for }\frac{1}{2}<s<1,\label{eq:tracec}
	\end{align}
\end{subequations}
see also \cite[Theorem 3.40]{McL00} or \cite[last paragraph in page 738]{NOS15}. 

Let $H^{-s}(\Omega)=(H_{0}^{s}(\Omega))'$ be the dual space of $H_{0}^{s}(\Omega)$. Using \cite[Theorem 9.31 and Remark 28]{Bre11}, there exists a sequence of complete eigenvalues $\{\lambda_{k}\}_{k=0}^{\infty}$ of the Laplacian operator $-\Delta$ with zero Dirichlet boundary condition, with the corresponding $L^{2}$ complete orthonormal basis $\{\phi_{k}\}_{k=0}^{\infty}$ in $H_{0}^{1}(\Omega)\cap\mathcal{C}^{\infty}(\Omega)$ such that 
\[
-\Delta \phi_{k} = \lambda_{k} \phi_{k}\quad\text{for all }k=0,1,2,\cdots.
\]
Therefore, for $v,w \in \mathscr{D}=\{v \in L^{2}(\Omega) : \sum_{k=0}^{\infty} \lambda_k v_{k}^{2} < \infty\}$, we have 
\[
_{\mathscr{D}'} \langle -\Delta v , w \rangle_{\mathscr{D}} = \sum_{k=0}^{\infty} \lambda_{k} v_{k} w_{k}\quad\text{if } v = \sum_{k=0}^{\infty} v_{k} \phi_{k}\text{ and } w = \sum_{k=0}^{\infty} w_{k} \phi_{k}.
\]
For the sake of simplicity, we simply denote 
\begin{equation}
-\Delta v=\sum_{k=0}^{\infty}\lambda_{k}v_{k}\phi_{k} \,\quad  \text{if } v=\sum_{k=0}^{\infty}v_{k}\phi_{k} \in \mathscr{D}. \label{eq:classical-Laplacian}
\end{equation}
Using the Lax-Milgram theorem and Poincar\'{e} inequality, we know that
\[
-\Delta:H_{0}^{1}(\Omega)\rightarrow H^{-1}(\Omega)\quad\text{is an isomorphism.}
\]
In view of \eqref{eq:classical-Laplacian}, we can define the spectral fractional Laplacian $(-\Delta)^{s}$ as 
\begin{equation}
(-\Delta)^{s}v:=\sum_{k=0}^{\infty}\lambda_{k}^{s}v_{k}\phi_{k} \quad  \text{if } v=\sum_{k=0}^{\infty}v_{k}\phi_{k} \in \mathscr{E}, \label{eq:spectral-definition}
\end{equation}
where $\mathscr{E}=\{v \in L^{2}(\Omega) : \sum_{k=0}^{\infty} \lambda_{k}^{s} v_{k}^{2} < \infty\}$. Actually, the operator $(-\Delta)^{s}$ above is also well-known as the spectral (Navier) fractional Laplacian; see, e.g., \cite{MN14}. In \cite{CS16}, using the equivalence of the Caffarelli-Silvestre-type extension, similar to that in \cite{CS07}, and the spectral definition \eqref{eq:spectral-definition}, Caffarelli and Stinga showed that $\mathscr{D}((-\Delta)^{s})=H_{0}^{s}(\Omega)$; see also \cite[equation (2.13)]{NOS15}, as well as \cite{MN14,MN16}. Indeed, \cite[eq. (1.2) and the line before Sec. 2.2]{CS16} shows that 
\[
(-\Delta)^{s}:H_{0}^{s}(\Omega)\rightarrow H^{-s}(\Omega)\quad\text{is an isomorphism.}
\]
This isomorphism can also be obtained using the Lax-Milgram theorem, with the bilinear form given in \eqref{eq:duality-pair} on the Hilbert space $H_{0}^{s}(\Omega)$. In other words, we have the following lemma: 
\begin{lem}
	\label{lem:solvability} Given any $f\in H^{-s}(\Omega)$, there exists a unique $v\in H_{0}^{s}(\Omega)$ that satisfies the fractional Poisson equation $(-\Delta)^{s}v=f$ in $\Omega$. 
\end{lem}

Consider the solution $V=V(x,y):\Omega\times[0,\infty)\rightarrow\mathbb{R}$
to the extension problem 
\begin{equation}
\begin{cases}
\Delta_{x}V+\frac{1-2s}{y}V_{y}+V_{yy}=0 & \text{in }\Omega\times(0,\infty),\\
V(x,y)=0 & \text{on }\partial\Omega\times[0,\infty),\\
V(x,0)=v(x) & \text{on }\Omega.
\end{cases}\label{eq:extension-problem}
\end{equation}
In \cite{CS16}, the authors showed that 
\begin{equation}
-\lim_{y\rightarrow0_{+}}y^{1-2s}V_{y}(x,y)=c_{s}(-\Delta)^{s}v(x)\quad\text{in }H^{-s}(\Omega),\quad\text{where }c_{s}=\frac{\Gamma(1-s)}{4^{s-1/2}\Gamma(s)}>0.\label{eq:CS-ext}
\end{equation}
In \cite[Theorem 2.5]{CS16}, we have 
\[
\int_{\Omega}\int_{0}^{\infty}y^{1-2s}|\nabla_{(x,y)}V(x,y)|^{2}\,dy\,dx=\|(-\Delta)^{s/2}v\|_{L^{2}(\Omega)}^{2},
\]
and also 
\begin{equation}
\|\bullet\|_{H_{0}^{s}(\Omega)}\text{ is equivalent to the norm }\|(-\Delta)^{s/2}\bullet\|_{L^{2}(\Omega)}.\label{eq:norm-equiv}
\end{equation}
Denote the $H^{-s}(\Omega) \times H_{0}^{s}(\Omega)$ duality pair by $\langle\bullet,\bullet\rangle={}_{H^{-s}(\Omega)}\langle\bullet,\bullet\rangle_{H_{0}^{s}(\Omega)}$, and we also have the following norm equivalence   
\begin{equation}
\|v\|_{H_{0}^{s}(\Omega)}^{2}\cong\|(-\Delta)^{s/2}v\|_{L^{2}(\Omega)}^{2}=\langle(-\Delta)^{s}v,v\rangle\quad\text{for all }v \in H_{0}^{s}(\Omega).\label{eq:duality-pair}
\end{equation}
We define 
\[
\mathscr{E}^{s}(V) := \int_{\Omega} \int_{0}^{\infty} y^{1-2s} |\nabla_{x,y}V(x,y)|^{2} \,dy\,dx
\]
and for each $v \in H_{0}^{s}(\Omega)$, we consider the space 
\[
\mathscr{W}^{s}(v) := \bigg\{ V \in H_{\rm loc}^{1}(\Omega \times (0,\infty)) : \mathscr{E}^{s}(V)<\infty,V|_{y=0}=v\text{ on }\Omega,V=0 \text{ on } \partial \Omega \times (0,\infty) \bigg\}.
\]
Indeed, the condition $V|_{y=0}=v$ on $\Omega$ in the definition of $\mathscr{W}^{s}(v)$ is defined by \cite[Proposition 2.5]{NOS15}. The following can also be found in \cite{ST10}: 
\begin{lem}\label{lem:minimizing-problem}
	For $v\in H_{0}^{s}(\Omega)$, there exists a unique minimizer $V(v)$
	\[
	V(v) = \underset{V \in \mathscr{W}^{s}(v)}{\rm argmin} \, \mathscr{E}^{s}(V).
	\]
	Moreover, $\langle (-\Delta)^{s}v,v \rangle = \mathscr{E}^{s}(V(v))$ holds. 
\end{lem}
\begin{rem}
$V(v)$ is a weak solution to the extension problem \eqref{eq:extension-problem} and \eqref{eq:CS-ext} holds. 
\end{rem}	
The norm equivalence \eqref{eq:duality-pair} suggests the following definition of the linear space 
\[
X_{0}^{2s}(\Omega):=\{u\in H_{0}^{s}(\Omega):(-\Delta)^{s}u\in L^{2}(\Omega)\} \quad \text{for } s \in (0,1),
\]
equipped with the norm 
\begin{equation}
\|v\|_{X_{0}^{2s}(\Omega)}:=\|(-\Delta)^{s}v\|_{L^{2}(\Omega)}^{2}.\label{eq:def-norm}
\end{equation}
\begin{lem}
The space $X_{0}^{2s}(\Omega)$ (equipped with the norm given in \eqref{eq:def-norm}) is complete. 
\end{lem}
\begin{proof}
Let $u_{n}$ be a Cauchy sequence in $X_{0}^{2s}(\Omega)$. Using the fact that $(-\Delta)^{s} = (-\Delta)^{s/2}(-\Delta)^{s/2}$, we know that $(-\Delta)^{s/2}u_{n}$ is also a Cauchy sequence in the Hilbert space $H_{0}^{s}(\Omega)$. Then, there exists a $v\in H_{0}^{s}(\Omega)$ such that $(-\Delta)^{s/2}u_{n} \rightarrow v$ in $H_{0}^{s}(\Omega)$, which implies that 
\[
\| (-\Delta)^{s}u_{n} - (-\Delta)^{s/2}v \|_{L^{2}(\Omega)} \rightarrow 0
\]
as $n \rightarrow \infty$. Because $(-\Delta)^{s/2}v \in L^{2}(\Omega) \subset H^{-s}(\Omega)$, using Lemma~\ref{lem:solvability}, there exists a unique $u \in H_{0}^{s}(\Omega)$ such that $(-\Delta)^{s}u=(-\Delta)^{s/2}v$. Hence, 
\begin{align*}
\|u_{n} - u\|_{X_{0}^{2s}(\Omega)} &= \| (-\Delta)^{s}(u_{n} - u) \|_{L^{2}(\Omega)} \\
& = \| (-\Delta)^{s} u_{n} - (-\Delta)^{s/2}v \|_{L^{2}(\Omega)} \\
& = \| (-\Delta)^{s/2}((-\Delta)^{s/2}u_{n} - v) \|_{L^{2}(\Omega)} \\
& \le C \| (-\Delta)^{s/2}u_{n} - v \|_{H_{0}^{s}(\Omega)} \rightarrow 0. 
\end{align*}
Therefore, $u_{n} \rightarrow u$ in $X_{0}^{2s}(\Omega)$; hence, we obtain the lemma.  
\end{proof}

See also \cite{CHM17,LM72} for more details regarding the fractional Sobolev spaces. However, we would like to emphasize that for the spectral fractional Laplacian (Navier fractional Laplacian) of order $\gamma \in \mathbb{R}_{+}$, 
\[
\{v \in H^{\gamma}(\mathbb{R}^{n}) : {\rm supp} \, (v) \subset \overline{\Omega}\} \subsetneq \mathscr{D}((-\Delta)_{{\rm Navier}}^{\gamma})
\]
for $\gamma \ge 3/2$, see e.g. \cite[Lemma 2]{MN16}. 

\begin{rem}\label{rem:equivalent-def-frac-space}
	In \cite{MN16,MN17}, the authors considered the space $\tilde{H}^{s}(\Omega) = \{ u \in H^{s}(\mathbb{R}^{n}) : u = 0 \text{ in } \mathbb{R}^{n} \setminus \overline{\Omega} \}$. The space $H^{s}(\mathbb{R}^{n})$ can be defined via Fourier transform; see also \cite[Proposition 3.4]{DPV12} for other equivalent definitions of the space $H^{s}(\mathbb{R}^{n})$. Indeed, $\tilde{H}^{s}(\Omega)=H_{0}^{s}(\Omega)$ for all $s\in (0,1)$ with $s\neq1/2$; see \cite[Proposition 1]{MN16}. For $s=1/2$, we know that $\tilde{H}^{1/2}(\Omega)$ is dense in $H_{0}^{1/2}(\Omega)$ without assumption of \eqref{eq:Lions-Magnes-condition}. From \eqref{eq:Lions-Magnes-condition}, using the interpolation relation in \cite[equation (2.8) and the next unlabeled equation]{NOS15}, $\tilde{H}^{1/2}(\Omega)=H_{0}^{1/2}(\Omega)$. Here, our notation $H_{0}^{1/2}(\Omega)$ is the Lions-Magenes space, which is usually denoted by $H_{00}^{1/2}(\Omega)$. See also \cite[Theorem 3.33]{McL00} for the case when $s\neq1/2$; the necessity of \eqref{eq:Lions-Magnes-condition} is also illustrated by an example in \cite[Exercise 3.22]{McL00}. 
 
\end{rem}

\subsection{Lewy-Stampacchia type inequality (see, e.g., \cite{LS69})}

In this subsection, we state the main result of this study. For each $\lambda \ge 0$, we define
the unilateral constraint $K_{0}$ by 
\[
K_{0}:=\{v\in H_{0}^{s}(\Omega):v\ge\psi\text{ a.e. in }\Omega\},\quad\text{where }\psi\in H_{0}^{s}(\Omega),
\]
which is a closed convex subset of $H_{0}^{s}(\Omega)$. We define the mapping $A\in B(H_{0}^{s}(\Omega),H^{-s}(\Omega))$ by 
\[
\langle Av,w\rangle:=\int_{\Omega}\bigg[((-\Delta)^{s/2}v)((-\Delta)^{s/2}w) +\lambda vw\bigg]\,dx\quad\text{for all }v,w\in H_{0}^{s}(\Omega),
\]
that is, $A=(-\Delta)^{s}+\lambda$ in a weak form. Because
$(-\Delta)^{s}\in{\rm Isom}(H_{0}^{s}(\Omega),H^{-s}(\Omega))$, $A\in{\rm Isom}(H_{0}^{s}(\Omega),H^{-s}(\Omega))$
for all $\lambda\ge0$. 

For each $f\in H^{-s}(\Omega)$, we also define a functional $J$
on $H_{0}^{s}(\Omega)$ by 
\[
J(v):=\frac{1}{2}\langle Av,v\rangle-\langle f,v\rangle. 
\]
Now we further assume that 
\begin{equation}
f\in L^{2}(\Omega)\quad\text{and}\quad\psi\in H_{0}^{s}(\Omega).\label{eq:LS-assump1}
\end{equation}
Moreover, we also assume that 
\begin{equation}
(-\Delta)^{s}\psi\in\mathcal{M}(\overline{\Omega}) \quad \text{and} \quad [(-\Delta)^{s}\psi]_{+}\in L^{2}(\Omega),\label{eq:LS-assump2}
\end{equation}
where $\mathcal{M}(\overline{\Omega})$ denotes the set of signed
Radon measures on $\overline{\Omega}$, and $[\bullet]_{+}$ represents the positive part function,
given by $[\alpha]_{+}=\max\{\alpha,0\}$ for all $\alpha\in\mathbb{R}$. From \eqref{eq:LS-assump2}, we have 
\begin{equation*}
	\hat{f}:=A\psi\in\mathcal{M}(\overline{\Omega})\quad\text{and}\quad[\hat{f}]_{+}=[A\psi]_{+}\in L^{2}(\Omega),
\end{equation*}
More precisely, because $\mathcal{C}(\overline{\Omega})\cap H_{0}^{s}(\Omega)$
is dense in $H_{0}^{s}(\Omega)$, the assumption $\hat{f}\in\mathcal{M}(\overline{\Omega})$
means that there exists $\mu\in\mathcal{M}(\overline{\Omega})$ such
that 
\[
\langle\hat{f},w\rangle=\int_{\overline{\Omega}}w\,d\mu\quad\text{for all }w\in\mathcal{C}(\overline{\Omega})\cap H_{0}^{s}(\Omega).
\]
The assumption $[\hat{f}]_{+}\in L^{2}(\Omega)$ means that the positive part $\mu_{+}$ of $\mu$ is absolutely continuous (with respect to the Lebesgue measure) with an $L^{2}$-Radon-Nikodym derivative (density function). From the assumptions \eqref{eq:LS-assump1} and \eqref{eq:LS-assump2}, together with Lemma~\ref{lem:A1}, we know that $[\hat{f}-f]_{+}\in L^{2}(\Omega)$; hence, 
\begin{align}\label{maxff}
\max\{f,\hat{f}\}=[\hat{f}-f]_{+}+f\in L^{2}(\Omega).
\end{align}
Now we are ready to state our main result. 
\begin{thm}[Lewy-Stampacchia type inequality] \label{thm:LS-ineq}Assume that
$f\in H^{-s}(\Omega)$ and $\psi\in H_{0}^{s}(\Omega)$. Then, there exists a unique $u\in H_{0}^{s}(\Omega)$ that satisfies the following variational inequality 
\begin{equation}
\langle Au,v-u\rangle\ge\langle f,v-u\rangle\quad\text{for all }v\in K_{0}.\label{eq:LS-ineq1}
\end{equation}
If we additionally assume \eqref{eq:LS-assump1} and \eqref{eq:LS-assump2},
we have 
\begin{equation}
u\in X_{0}^{2s}(\Omega)\quad\text{and}\quad f\le Au\le\max\{f,\hat{f}\}\quad\text{a.e. in }\Omega.\label{eq:LS-ineq2}
\end{equation}
\end{thm}

We next give a comparison theorem for variational inequalities of obstacle type. 
\begin{thm}
	[Comparison principle for variational inequality]\label{thm:Comparison-var}
	For $i=1,2$, let $(f_{i},\psi_{i})\in H^{-s}(\Omega)\times H_{0}^{s}(\Omega)$,
	and set 
	\begin{align*}
		K_{0}^{i} & :=\{v\in H_{0}^{s}(\Omega):v\ge\psi_{i}\text{ a.e. in }\Omega\},\\
		K_{1}^{i} & :=\{v\in H_{0}^{s}(\Omega):Av\ge f_{i}\text{ in }H^{-s}(\Omega)\}.
	\end{align*}
	Let $u_{i}\in H_{0}^{s}(\Omega)$ be the unique solution of the variational
	inequality: 
	\[
	u_{i}\in K_{0}^{i}:\quad\langle Au_{i},v-u_{i}\rangle\ge\langle f_{i},v-u_{i}\rangle\quad\text{for all }v\in K_{0}^{i}.
	\]
	If $f_{1}\le f_{2}$ in $H^{-s}(\Omega)$ and $\psi_{1}\le\psi_{2}$
	a.e. in $\Omega$, then it holds that $u_{1}\le u_{2}$ a.e. in $\Omega$. 
\end{thm}

Indeed, $u$ satisfies the inequality \eqref{eq:LS-ineq1} if and only if 
\[
J(u)=\min_{v\in K_{0}}J(v).
\]
See Lemma~\ref{lem:LS1} for details.  

The inequality in \eqref{eq:LS-ineq2} is called the Lewy-Stampacchia inequality. In the case of the second-order elliptic operator including the standard Laplacian ($s=1$), it was proved in \cite[Theorem 3.1]{LS69} using a penalty method to show the $W^{2,p}$-regularity of a solution for a unilateral variational inequality. Please also see \cite{KS80}.  

In 1986, Gustafsson \cite{Gus86} provided an alternative simpler proof for the Lewy-Stampacchia inequality by introducing an equivalent variational inequality problem, which corresponds to Lemma~\ref{lem:LS2}(g) below. Gustafsson's method was employed in \cite{AK19} to prove the Lewy-Stampacchia inequality for variational inequality with a mixed boundary condition. They also used it to construct a strong solution to a unidirectional diffusion equation. In this study, we use Gustafsson's approach \cite{Gus86} to prove Theorem~\ref{thm:LS-ineq}. 

In \cite{AR18}, they consider a class of fractional obstacle type problem related to the fractional elliptic operator $(-\nabla \cdot A(x) \nabla + c(x))^{s}$. In addition, they also provide a solution algorithm for numerical implementation. A Lewy-Stampacchia inequality for $-D^{s} \cdot( A(x) D^{s} u)$, where $D^{s}$ is the distributional Riesz fractional gradient (a.k.a. $s$-gradient) and $D^{s} \cdot$ is the $s$-divergence, was derived in \cite{LR21}. Similar results were proved in \cite{MNS17}. Here we would like to point out that the operator considered in \cite{MNS17} is not equivalent to the one in this paper. 

The inequality \eqref{eq:LS-ineq1} for $s=1$ is the simplest form of the classical obstacle problem; see, e.g., \cite{KS80}. Here, we recall a classical example of obstacle problems: an elastic membrane, with vertical displacement $u$ on a domain $\Omega$, which is constrained
at its boundary $u=0$ along $\partial\Omega$ and is forced to
lie above some obstacle $u\ge\psi$. In \cite{SV13}, they considered a similar problem, but the fractional
operator they used is different from ours; see, e.g., \cite{DWZ19,MN14,MN16}. 

\subsection{Main difficulties}
For the classical Laplacian case in \cite{AK19}, the following observation played a crucial role: 
\begin{equation}
\langle -\Delta u_{+},u_{-} \rangle = \int_{\Omega} \nabla u_{+} \cdot \nabla u_{-} \,dx = 0 \label{eq:difficulties1}
\end{equation} 
for $u \in H_{0}^{1}(\Omega)$. However, in the fractional Laplacian case, the support of $(-\Delta)^{s} u_{+}$ is larger than that of $u_{+}$. This means \eqref{eq:difficulties1} is not valid. Therefore, we solve this problem using a technique from \cite{MN17}; see Lemma~\ref{lem:MN17}. 

The operator $(-\Delta)^{s}$ is nonlocal. Caffarelli, Silvestre, and Stinga \cite{CS07,CS16} localized the operator in a higher-dimensional space, gained some
regularity, and also verified the formal computations in a rigorous manner, which is crucial to our proof. See also \cite{Rul15,RW19},
which study some ``local'' properties of the ``nonlocal'' operator
$(-\Delta)^{s}$, using the Caffarelli-Silvestre-type extension.

\subsection{Well-Posedness of Anomalous Unidirectional Diffusion Equations}

We also can apply the Lewy-Stampacchia-type inequality above and modify
some ideas of \cite{AK19} to prove the well-posedness of anomalous unidirectional diffusion equations. For $s\in(0,1)$, we
consider the following initial-boundary value problem for the unidirectional
evolution equation of fractional-diffusion type: 
\begin{subequations}
\begin{align}
\partial_{t}u & =[-(-\Delta)^{s}u+f]_{+}\quad\text{in }\Omega\times(0,T],\label{eq:diff1}\\
u & =0\quad\text{on }\partial\Omega\times(0,T] \quad \text{if }s \in (1/2,1),\label{eq:diff-Dirichlet}\\
u|_{t=0} & =u_{0}\quad\text{in }\Omega,\label{eq:diff-initial}
\end{align}
\end{subequations}
where $\partial_{t}=\frac{\partial}{\partial t}$, $f=f(x,t)$ and $u_{0}=u_{0}(x)$ are given functions. If $u$ is not sufficiently regular, the trace may not be well-defined; see \eqref{eq:tracea}, \eqref{eq:traceb}, and \eqref{eq:tracec}. In this case, we simply omit the Dirichlet boundary condition \eqref{eq:diff-Dirichlet}. Because the fractional Laplacian is defined by the approximation of \eqref{eq:spectral-definition}, the boundary conditions will recover in the case when we recover the regularity of the solution $u$.  
For simplicity, here and after,
we denote $u(t)=u(\bullet,t)$ and $f(t)=f(\bullet,t)$, if no ambiguity
occurs. 

Now, we define the precise meaning of the strong solution to \eqref{eq:diff1},
\eqref{eq:diff-Dirichlet}, and \eqref{eq:diff-initial}: 
\begin{defn}
\label{def:sol}For a given $f\in L^{2}(0,T;L^{2}(\Omega))$
and $u_{0}\in H_{0}^{s}(\Omega)$, we say that $u$ is a \emph{strong
solution} to \eqref{eq:diff1}, \eqref{eq:diff-Dirichlet}, and \eqref{eq:diff-initial}, if 
\begin{enumerate}
\item [(a)] $u\in H^{1}(0,T;L^{2}(\Omega))\cap L^{2}(0,T;X_{0}^{2s}(\Omega))$, 
\item [(b)] $\partial_{t}u(t)=[-(-\Delta)^{s}u(t)+f(t)]_{+}$ in $L^{2}(\Omega)$ for a.e. $t\in(0,T)$; 
\item [(c)] $u(0)=u_{0}$. 
\end{enumerate}
\end{defn}

\begin{rem}
The operator $[\bullet]_{+}$ is defined in a pointwise manner. The $L^{2}(\Omega)$
assumption can guarantee that the function is finite a.e. in $\Omega$. Therefore, it is crucial
to show $(-\Delta)^{s}u(t)\in L^{2}(\Omega)$ for a.e. $t\in(0,T)$, that is, $u(t)\in X_{0}^{2s}(\Omega)$ for a.e. $t\in(0,T)$.
From the inclusion $H^{1}(0,T;L^{2}(\Omega))\subset\mathcal{C}([0,T];L^{2}(\Omega))$, the initial data at $t=0$ are well defined. 
\end{rem}

First of all, we state the uniqueness and stability of the strong
solution to \eqref{eq:diff1}, \eqref{eq:diff-Dirichlet}, and \eqref{eq:diff-initial}. 
\begin{thm}
[Uniqueness and stability results] \label{thm:Uniqueness}Given $f\in L^{2}(0,T;L^{2}(\Omega))$
and $u_{0}\in H_{0}^{s}(\Omega)$. Then, the strong solution to \eqref{eq:diff1},
\eqref{eq:diff-Dirichlet}, and \eqref{eq:diff-initial} is unique,
and the solution is continuous with respect to $f$ and $u_{0}$ in the following sense: 
\begin{align}
& \|u_{1}-u_{2}\|_{H^{1}(0,T;L^{2}(\Omega))\cap L^{\infty}(0,T;H_{0}^{s}(\Omega))} \nonumber \\
\le & C\bigg[\|u_{1}(\bullet,0)-u_{2}(\bullet,0)\|_{H_{0}^{s}(\Omega)}+\|f_{1}-f_{2}\|_{L^{2}(0,T;L^{2}(\Omega))}\bigg], \label{eq:conti-sol}
\end{align}
where $u_{1},u_{2}$ are strong
solutions of \eqref{eq:diff1}, \eqref{eq:diff-Dirichlet}, and \eqref{eq:diff-initial},
with initial conditions $u_{1,0},u_{2,0}\in H_{0}^{s}(\Omega)$, respectively. 
\end{thm}

By discretizing the time variable using the implicit Euler scheme, together
with the compact embedding 
\begin{equation}
H_{0}^{s}(\Omega)\xhookrightarrow{\text{compact}}L^{2}(\Omega)\quad\text{for all }s\in(0,1),\label{eq:cpt-emb}
\end{equation}
see, e.g., \cite{DPV12} and Ascoli's compactness lemma, we can
obtain the following existence result. 
\begin{thm}
[Existence results] \label{thm:Existence}Assume that the initial
datum $u_{0}\in H_{0}^{s}(\Omega)$ satisfies 
\begin{equation}
(-\Delta)^{s}u_{0}\in\mathcal{M}(\overline{\Omega})\quad\text{and}\quad[(-\Delta)^{s}u_{0}]_{+}\in L^{2}(\Omega),\label{eq:exist1}
\end{equation}
while the external force $f\in L^{2}(0,T;L^{2}(\Omega))$ satisfies
the obstacle condition 
\begin{equation}
f(x,t)\le f^{*}(x)\quad\text{a.e. in }\Omega\times(0,T)\label{eq:exist2}
\end{equation}
for some $f^{*}\in L^{2}(\Omega)$. Then there exists a strong solution
to \eqref{eq:diff1}, \eqref{eq:diff-Dirichlet}, and \eqref{eq:diff-initial}. 
\end{thm}

The connection between Theorem~\ref{thm:LS-ineq} and Theorem~\ref{thm:Existence} can be seen in the proof of Lemma~\ref{lem:discretization}. 

Then, employing the same arguments as in \cite{AK19}, we can obtain
the following comparison principle and identify the limit of each
solution $u=u(x,t)$ as $t\rightarrow\infty$. 
\begin{thm}
[Comparison principle] \label{thm:Comparison}Suppose that the initial
data $u_{0,1},u_{0,2}\in H_{0}^{s}(\Omega)$ of $u_{1},u_{2}$, respectively,
both satisfy \eqref{eq:exist1}, and the external forces $f_{1},f_{2}\in L^{2}(0,T;L^{2}(\Omega))$
of $u_{1},u_{2}$, respectively, both satisfy the obstacle condition
\eqref{eq:exist2}. If 
\[
u_{0,1}\le u_{0,2}\quad\text{a.e. in }\Omega
\]
and 
\[
f_{1}\le f_{2}\quad\text{a.e. in }\Omega\times(0,T),
\]
then $u_{1}\le u_{2}$ a.e. in $\Omega\times(0,T)$. 
\end{thm}

\begin{thm}
[Convergence of solutions as $t\rightarrow \infty$] \label{thm:Asymp}Let
$u_{0}\in X_{0}^{2s}(\Omega)$. Suppose that there exists a function
$f_{\infty}\in L^{2}(\Omega)$ such that $f-f_{\infty}\in L^{2}(0,\infty;L^{2}(\Omega))$.
Moreover, assume that $f\in L^{\infty}(0,\infty;L^{2}(\Omega))$ with
\eqref{eq:exist2}. Then, the unique solution $u=u(x,t)$ of \eqref{eq:diff1},
\eqref{eq:diff-Dirichlet}, and \eqref{eq:diff-initial} on $[0,\infty)$,
that is, on $[0,T]$ for each $T>0$, converges to a function $u_{\infty}=u_{\infty}(x)\in X_{0}^{2s}(\Omega)$, strongly in $H_{0}^{s}(\Omega)$ as $t\rightarrow\infty$. Moreover,
the limit $u_{\infty}$ satisfies 
\[
u_{\infty}\ge u_{0}\quad\text{and}\quad(-\Delta)^{s}u_{\infty}\ge f_{\infty}\quad\text{a.e. in }\Omega.
\]
In addition, if $f(x,t)\le f_{\infty}(x)$ for a.e. $(x,t)\in\Omega\times(0,\infty)$,
then the limit $u_{\infty}$ coincides with the unique solution $\overline{u}_{\infty}\in X_{0}^{2s}(\Omega)\cap K_{0}(u_{0})$
of the following variational inequality: 
\[
\int_{\Omega}((-\Delta)^{s/2}\overline{u}_{\infty})((-\Delta)^{s/2}(v-\overline{u}_{\infty}))\,dx\ge\int_{\Omega}f_{\infty}(v-\overline{u}_{\infty})\,dx\quad\text{for all }v\in K_{0}(u_{0}),
\]
where $K_{0}(u_{0}):=\{v\in H_{0}^{s}(\Omega):v\ge u_{0}\text{ a.e. in }\Omega\}$. 
\end{thm}

\subsection{Some remarks on the fractional Laplacian}

For the $\mathbb{R}^{n}$ case, there are at least 10 equivalent definitions; see, e.g., an interesting survey paper \cite{Kwa17}. The easiest way to define $(-\Delta)^{s}$ is simply using the Fourier transform. However, this definition is non-local. Thanks to the Caffarelli-Silvestre extension \cite{CS07}, the operator can be localized in an extended half-space, so we can obtain some further results. Moreover, the operator $(-\Delta)^{s}$ can be defined as the generator (see, e.g., \cite{Sch14}), or by the Dynkin characteristic operator of the isotropic $2s$-stable L\'{e}vy process, and it is widely used in probability theory, for example, in the continuous time random walk (see, e.g., Chapter 4 of \cite{KRS08}). 

The operator $(-\Delta)^{s}$ in the one-dimensional case can be represented by the Riemann-Liouville derivatives, as well as the Caputo derivatives. Each of the definitions has different advantages and drawbacks. Therefore, the equivalence between them is extremely important. Here, we would like to emphasize that the equivalence is in the sense of the norm, but they are not pointwise equivalent. There are some counterexamples in Kwa\'{s}nicki's paper \cite{Kwa17}. 

The Fourier transform in the $\mathbb{R}^{n}$ case can be referred to as a continuous spectrum. Therefore, in the case of bounded domain $\Omega$, we can define $(-\Delta)^{s}$ using the discrete $L^{2}$-spectrum of $-\Delta$ using a similar idea to that used in the $\mathbb{R}^{n}$ case, as shown in \eqref{eq:spectral-definition}. $(-\Delta)^{s}$ is non-local, that is, the definition of $(-\Delta)^{s}$ depends on the entire $\Omega$ as well as the Dirichlet data on $\partial\Omega$. Thus, in general, $(-\Delta)^{s}$ defined on a bounded domain $\Omega$ is not equivalent to the restriction of $(-\Delta)^{s}$ in the $\mathbb{R}^{n}$ case. However, we can still use some techniques from the $\mathbb{R}^{n}$ case. The most noteworthy one is the Caffarelli-Silvestre-type extension \cite{CS16}. This enables us to obtain the appropriate regularity for $(-\Delta)^{s} u$, which is the key to our proof. Moreover, it is easy to see that the $\mathbb{R}^{n}$-fractional Laplacian is a pseudo-differential operator. Indeed, $(-\Delta)^{s}$ on a bounded domain is also a pseudo-differential operator (see, e.g., \cite{Gru15}). Although the definition of fractional Laplacian is quite abstract, it can be computed numerically, as in \cite{AG17,DWZ18}.

\subsection{Organization of the paper}

Theorem~\ref{thm:LS-ineq} is given in Section~\ref{sec:Proof-of-LW-ineq}, Theorem~\ref{thm:Comparison} is proved in Section~\ref{sec:proof-Comparison}, and Theorem~\ref{thm:Uniqueness} is proved in Section~\ref{sec:Proof-of-Uniqueness}. Because the proof of Theorem~\ref{thm:Existence} is very similar to that given in \cite{AK19}, we only sketch it. We omit the proof of Theorem~\ref{thm:Comparison} and Theorem~\ref{thm:Asymp} because they can be easily proved using the same ideas in \cite{AK19}. Finally, we list some auxiliary lemmas in Appendix~\ref{sec:Aux-Lemma}, and in Appendix~\ref{sec:chain-rule}, we provide a note on a chain rule.  

\section{\label{sec:Proof-of-LW-ineq}Proof of Theorem~\ref{thm:LS-ineq}}

Given any $\lambda \ge 0$, we define the following unilateral constraint 
\[
K_{1}:=\{v\in H_{0}^{s}(\Omega):Av\ge f\text{ in }H^{-s}(\Omega)\},
\]
that is, $\langle Av-f,\varphi\rangle\ge0$ for all $\varphi\in H_{0}^{s}(\Omega)$
with $\varphi\ge0$ a.e. in $\Omega$. Moreover, we define $\hat{J}:H_{0}^{s}(\Omega)\rightarrow\mathbb{R}$
by 
\[
\hat{J}(v):=\frac{1}{2}\langle Av,v\rangle-\langle\hat{f},v\rangle\quad\text{for }v\in H_{0}^{s}(\Omega).
\]

\begin{lem}
\label{lem:LS1}Let $\lambda \ge 0$, $f\in H^{-s}(\Omega)$ and $\psi\in H_{0}^{s}(\Omega)$.
Then, there exists a unique $u\in H_{0}^{s}(\Omega)$ that satisfies
the following five equivalent conditions: 
\begin{enumerate}
\item [(a)] $u\in K_{0}$, $J(u)\le J(v)$
for all $v\in K_{0}$; 
\item [(b)] $u\in K_{0}$, $\langle Au,v-u\rangle\ge\langle f,v-u\rangle$
for all $v\in K_{0}$; 
\item [(c)] $u\in K_{0}\cap K_{1}$, $\langle Au-f,u-\psi\rangle=0$; 
\item [(d)] $u\in K_{1}$, $\langle Au,v-u\rangle\ge\langle\hat{f},v-u\rangle$
for all $v\in K_{1}$; 
\item [(e)] $u\in K_{1}$, $\hat{J}(u)\le\hat{J}(v)$
for all $v\in K_{1}$. 
\end{enumerate}
\end{lem}

\begin{proof}[Proof of Lemma~{\rm \ref{lem:LS1}}]
It is easy to see that $J$ is continuous and strictly convex on $K_{0}$, which immediately implies the uniqueness of $u\in H_{0}^{s}(\Omega)$. According to Stampacchia's theorem (see, e.g., \cite[Theorem 5.6]{Bre11}), it is well known that (a)$\iff$(b) and (d)$\iff$(e), and therefore $u\in H_{0}^{s}(\Omega)$ exists (indeed, it can also be proved using a standard argument given in \cite{Eva98}). Therefore, we only need to show the equivalence of (a)--(e). 

First, we show that (b)$\implies$(c). Condition (b) can be equivalently rewritten as 
\begin{equation}
	\langle Au-f,v-u\rangle\ge0\quad\text{for all }v\in K_{0}.\label{eq:LS-equiv1}
\end{equation}
For any $\varphi\in H_{0}^{s}(\Omega)$ with $\varphi\ge0$ a.e. in
$\Omega$, substituting $v=u+\varphi\in K_{0}$ into \eqref{eq:LS-equiv1},
we have 
\[
\langle Au-f,\varphi\rangle\ge0,
\]
i.e., $Au\ge f$ in $H^{-s}(\Omega)$, which yields
that $u\in K_{1}$. Conversely, substituting
$v=\psi\in K_{0}$ and $v=2u-\psi\in K_{0}$ to \eqref{eq:LS-equiv1},
we reach 
\[
\langle Au-f,\psi-u\rangle\ge0\quad\text{and}\quad\langle Au-f,u-\psi\rangle\ge0.
\]
Then, we obtain $\langle Au-f,\psi-u\rangle=0$,
and hence condition (c) holds. 

Next, we prove the converse, (c)$\implies$(b). For any $v\in K_{0}$,
we see that 
\[
\langle Au,v-u\rangle-\langle f,v-u\rangle=\overbrace{\langle Au-f,v-\psi\rangle}^{\ge0\text{ by definition of }K_{1}}-\overbrace{\langle Au-f,u-\psi\rangle}^{=0\text{ by (c)}}\ge0,
\]
which shows that condition (b) holds. 

Finally, we prove the equivalence between (c) and (d) in a similar
fashion to that above. We first show that (c)$\implies$(d). Using the symmetry
of $\langle A\bullet,\bullet\rangle$, note that 
\begin{align*}
	\langle Au,v-u\rangle-\langle\hat{f},v-u\rangle= & \langle Av-Au,u\rangle-\langle Av-Au,\psi\rangle\\
	= & \overbrace{\langle Av-f,u-\psi\rangle}^{\ge0\text{ by definition of }K_{1}}-\overbrace{\langle Av-f,u-\psi\rangle}^{=0\text{ by (c)}}\ge0.
\end{align*}
This is the desired condition (d). 

Conversely, we shall prove that (d)$\implies$(c). We write (d) as 
\begin{equation}
	\langle Av-Au,u-\psi\rangle\ge0\quad\text{for all }v\in K_{1}.\label{eq:LS-equiv2}
\end{equation}
For any $\varphi\in L^{2}(\Omega)$ with $\varphi\ge0$ a.e. in $\Omega$,
substituting $v=u+A^{-1}\varphi\in K_{1}$
to \eqref{eq:LS-equiv2}, we have 
\[
(\varphi,u-\psi)_{L^{2}(\Omega)}\ge0.
\]
By the arbitrariness of $\varphi\ge0$, we reach $u\in K_{0}$.
Moreover, we substitute $v=A^{-1}f\in K_{1}$
and $v=2u-A^{-1}f\in K_{1}$, and similarly,
we obtain 
\[
\langle Au-f,u-\psi\rangle=0,
\]
which verifies condition (c).
\end{proof}

The following lemma can be found in \cite[Remark 3.3]{MN17}. 
\begin{lem}\label{lem:MN17}
If $u \in H_{0}^{s}(\Omega)$, then $u_{\pm} \in H_{0}^{s}(\Omega)$ and 
\[
\langle (-\Delta)^{s} u_{+} , u_{-} \rangle \le 0,
\]
where $\alpha_{-} = - \min\{\alpha,0\}$ for all $\alpha \in \mathbb{R}$. 
\end{lem}

\begin{proof}
Because $u \in H_{0}^{s}(\Omega)$, using Lemma~\ref{lem:A3}, we know that $u_{\pm} \in H_{0}^{s}(\Omega)$. From Lemma~\ref{lem:minimizing-problem}, 
\begin{align*}
\langle (-\Delta)^{s}u,u\rangle &= \mathscr{E}^{s}(V(u)) \\
&=  \mathscr{E}^{s}(V(u)_{+}) + \mathscr{E}^{s}(V(u)_{-}) \\
&\ge \mathscr{E}^{s}(V(u_{+})) + \mathscr{E}^{s}(V(u_{-})) \\
&= \langle (-\Delta)^{s}u_{+},u_{+}\rangle + \langle (-\Delta)^{s}u_{-},u_{-}\rangle.
\end{align*}
Hence, we have 
\[
2\langle (-\Delta)^{s}u_{+},u_{-}\rangle = \langle (-\Delta)^{s}u_{+},u_{+}\rangle + \langle (-\Delta)^{s}u_{-},u_{-}\rangle - \langle (-\Delta)^{s}u,u\rangle \le 0,
\]
which is our desired lemma. 
\end{proof}

Under conditions \eqref{eq:LS-assump1} and \eqref{eq:LS-assump2}, for each $\lambda \ge 0$, 
we introduce the obstacle set 
\[
K_{2}:=\{v\in H_{0}^{s}(\Omega):f\le Av\le\max\{f,\hat{f}\}\text{ in }H^{-s}(\Omega)\}\subset K_{1},
\]
which is a closed convex subset of $H_{0}^{s}(\Omega)$. 

\begin{lem}
\label{cor:inclusion}Suppose that \eqref{eq:LS-assump1} and \eqref{eq:LS-assump2} hold. Then $K_{2}\subset X_{0}^{2s}(\Omega)$. 
\end{lem}
\begin{proof}
	Let $v\in K_{2}$; then, we have 
	\[
	\langle f,\varphi\rangle\le\langle Av,\varphi\rangle\le\langle g,\varphi\rangle\quad\text{for all }\varphi\in H_{0}^{s}(\Omega)\text{ with }\varphi\ge0,
	\]
	where $g=\max\{f,\hat{f}\}$. Given any $\phi\in H_{0}^{s}(\Omega)$,
	by Lemma~\ref{lem:A3}, we have $\phi_{\pm}\in H_{0}^{s}(\Omega)$.
	Therefore, we have 
	\begin{align}
		\langle Av,\phi\rangle & =\langle Av,\phi_{+}\rangle-\langle Av,\phi_{-}\rangle\nonumber \\
		& \le\langle g,\phi_{+}\rangle-\langle f,\phi_{-}\rangle\nonumber \\
		& \le\|g\|_{L^{2}(\Omega)}\|\phi_{+}\|_{L^{2}(\Omega)}+\|f\|_{L^{2}(\Omega)}\|\phi_{-}\|_{L^{2}(\Omega)}\nonumber \\
		& \le\bigg(\|f\|_{L^{2}(\Omega)}+\|g\|_{L^{2}(\Omega)}\bigg)\|\phi\|_{L^{2}(\Omega)}.\label{eq:patch-ver11-1}
	\end{align}
	Similarly, we also have 
	\begin{align}
		-\langle Av,\phi\rangle & =-\langle Av,\phi_{+}\rangle+\langle Av,\phi_{-}\rangle\nonumber \\
		& \le-\langle f,\phi_{+}\rangle+\langle g,\phi_{-}\rangle\nonumber \\
		& \le\|f\|_{L^{2}(\Omega)}\|\phi_{+}\|_{L^{2}(\Omega)}+\|g\|_{L^{2}(\Omega)}\|\phi_{-}\|_{L^{2}(\Omega)}\nonumber \\
		& \le\bigg(\|f\|_{L^{2}(\Omega)}+\|g\|_{L^{2}(\Omega)}\bigg)\|\phi\|_{L^{2}(\Omega)}.\label{eq:patch-ver11-2}
	\end{align}
	Combining \eqref{eq:patch-ver11-1} and \eqref{eq:patch-ver11-2},
	we have 
	\[
	|_{H^{-s}(\Omega)}\langle Av,\phi\rangle_{H_{0}^{s}(\Omega)}|\le\bigg(\|f\|_{L^{2}(\Omega)}+\|g\|_{L^{2}(\Omega)}\bigg)\|\phi\|_{L^{2}(\Omega)}\quad\text{for all }\phi\in H_{0}^{s}(\Omega).
	\]
	Therefore, by the Hahn-Banach theorem, there exists $F\in(L^{2}(\Omega))'$
	such that 
	\[
	_{(L^{2}(\Omega))'}\langle F,\phi\rangle_{L^{2}(\Omega)}={}_{H^{-s}(\Omega)}\langle Av,\phi\rangle_{H_{0}^{s}(\Omega)}\quad\text{for all }\phi\in H_{0}^{s}(\Omega).
	\]
	Following Riesz's representation theorem, there exists $G\in L^{2}(\Omega)$
	such that 
	\[
	(G,\phi)_{L^{2}(\Omega)}={}_{(L^{2}(\Omega))'}\langle F,\phi\rangle_{L^{2}(\Omega)}={}_{H^{-s}(\Omega)}\langle Av,\phi\rangle_{H_{0}^{s}(\Omega)}\quad\text{for all }\phi\in H_{0}^{s}(\Omega).
	\]
	Therefore, we can identify $Av$ and $G\in L^{2}(\Omega)$, that is,
	$Av\in L^{2}(\Omega)$, Then, we conclude that $(-\Delta)^{s}v\in L^{2}(\Omega)$.
	Because $v\in H_{0}^{s}(\Omega)$, therefore, $v\in X_{0}^{2s}(\Omega)$. 
\end{proof}
 
\begin{rem}
\label{rem:Laplace} Indeed, from Lemma~\ref{lem:solvability}, given any $\lambda \ge 0$, $A|_{X_{0}^{2s}(\Omega)}:X_{0}^{2s}(\Omega)\rightarrow L^{2}(\Omega)$
is an isomorphism. 
\end{rem}

\begin{lem}
	\label{lem:LS2}Let $\lambda \ge 0$. Suppose that \eqref{eq:LS-assump1} and \eqref{eq:LS-assump2} hold. Then, the following conditions are equivalent to the conditions
	in Lemma~{\rm \ref{lem:LS1}}: 
	\begin{enumerate}
		\item [(f)] $u\in K_{2}$, $\hat{J}(u)\le\hat{J}(v)$
		for all $v\in K_{2}$; 
		\item [(g)] $u\in K_{2}$, $\langle Au,v-u\rangle\ge\langle\hat{f},v-u\rangle$
		for all $v\in K_{2}$; 
		\item [(h)] $u\in K_{0}\cap K_{2}$, $(Au-f)(u-\psi)=0$
		a.e. in $\Omega$.
	\end{enumerate}
\end{lem} 

\begin{proof}[Proof of Lemma~{\rm \ref{lem:LS2}}]
  According to Stampacchia's theorem (see, e.g., \cite[Theorem 5.6]{Bre11}), it is well known that (f)$\iff$(g) and there exists a unique $u\in K_{2}$ that satisfies (f) and (g).

From the condition (g) and the symmetry of $\langle A\bullet,\bullet\rangle$,
it follows that 
\begin{align*}
0 \le\langle Au-\hat{f},v-u\rangle
=\langle A(u-\psi),v-u\rangle
=\langle A(v-u),u-\psi\rangle ,
\end{align*}
for all $v\in K_{2}$.
Setting $w:=u-\psi\in H_{0}^{s}(\Omega)$, we have
\begin{align}
 & \langle A(v-u),w\rangle \ge 0\quad\text{for all }v\in K_{2}.\label{eq:LS-equiv5}
\end{align}

We define a measurable set $N:=\{x\in\Omega:w(x)<0\}$
and a truncation of $Au$ by 
\[
g(x):=\begin{cases}
\max\{f,\hat{f}\} & \text{if }x\in N,\\
Au & \text{if }x\in\Omega\setminus N.
\end{cases}
\]
Using \eqref{eq:LS-assump1}, \eqref{maxff}, and Lemma~\ref{cor:inclusion},
we know that $g\in L^{2}(\Omega)$ and $A^{-1}g\in K_{2}$. Substituting
$v=A^{-1}g$ into \eqref{eq:LS-equiv5}, we have
\[
0\le\langle g-Au,w\rangle=\int_{N}(\max\{f,\hat{f}\}-Au)w\,dx.
\]
Because $\max\{f,\hat{f}\}-Au\ge0$
and $w<0$ in $N$, we reach 
\[
\max\{f,\hat{f}\}-Au=0\quad\text{a.e. in }N.
\]
Hence, we obtain
\begin{align*}
  \langle Aw, w_-\rangle
  &=\langle Au-\max\{f,\hat{f}\}, w_-\rangle +\langle \max\{f,\hat{f}\}-\hat{f}, w_-\rangle\\
  &=\int_{N}(\max\{f,\hat{f}\}-Au)w\,dx +\langle [\hat{f}-f]_-, w_-\rangle\\
  &=\langle [\hat{f}-f]_-, w_-\rangle ,
\end{align*}
where we remark that $\hat{f}-f\in \mathcal{M}(\overline{\Omega})$, and its negative part is denoted by $[\hat{f}-f]_-\in \mathcal{M}_+(\overline{\Omega})$,
as
\[
f\in L^2(\Omega),\quad \hat{f}\in H^{-s}(\Omega)\cap \mathcal{M}(\overline{\Omega}),
\quad [\hat{f}]_+\in L^2(\Omega).
\]
Furthermore, from \eqref{maxff}, it follows that
\[
[\hat{f}-f]_-=[\hat{f}-f]_+-(\hat{f}-f)=\max\{f,\hat{f}\}-\hat{f}\in H^{-s}(\Omega).
\]

Applying Lemma~\ref{densitylemma},
we can choose $\{ v_{m}\}_{m=1}^\infty \subset \mathcal{C}_{c}^{\infty}(\Omega)$ such that
$v_{m}(x)\ge 0$ for $x\in\Omega$ and $v_{m}\to w_-$ in $H_0^s(\Omega)$ as $m\to \infty$.
Because $\langle [\hat{f}-f]_-, v_{m}\rangle \ge 0$, we obtain
\begin{align*}
  \langle [\hat{f}-f]_-, w_-\rangle
  =\lim_{m\to\infty}  \langle [\hat{f}-f]_-, v_m\rangle \ge 0.
\end{align*}
These inequalities and Lemma~\ref{lem:MN17} imply
\begin{align*}
  \langle Aw_-, w_-\rangle  =\langle Aw_+, w_-\rangle  -\langle Aw, w_-\rangle \le 0.
\end{align*}
Using the coercivity of $\langle A\bullet,\bullet\rangle$ in $H_{0}^{s}(\Omega)$,
we conclude that $w_{-}=0$, that is, $u\ge \psi$
a.e. in $\Omega$, which shows that $u\in K_{0}$.

Substituting $v=A^{-1}f\in K_2$ into \eqref{eq:LS-equiv5}, we obtain
$\langle Au-f, u-\psi\rangle \le 0$.
However, from $u\in K_0\cap K_1$, $\langle Au-f, u-\psi\rangle = 0$ holds
and $u$ turns out to be the unique solution to (c) of Lemma~\ref{lem:LS1}.
Hence, the equivalence of conditions (a)-(g) has been shown.

Furthermore, if $u$ satisfies conditions (a)-(g), 
$Au\in L^2(\Omega)$ follows from Lemma~\ref{cor:inclusion}
and condition (c) implies that
$(Au-f)(u-\psi)\ge 0$ a.e. in $\Omega$ and that
\[
\int_\Omega (Au-f)(u-\psi)\,dx =0.
\]
Hence, it follows that $u$ satisfies condition (h).
Because the implication (h) $\Longrightarrow$ (c) is trivial,
conditions (a)-(h) are equivalent.
\end{proof}
 
Now, we are ready to prove Theorem~\ref{thm:LS-ineq}. 
\begin{proof}
[Proof of Theorem~{\rm \ref{thm:LS-ineq}}] The existence and uniqueness of the solution $u$ of the variation inequality \eqref{eq:LS-ineq1} is guaranteed by Lemma~\ref{lem:LS1}. Next, by Lemma~\ref{lem:LS2}, we know that $u\in K_{2}$. By Lemma~\ref{cor:inclusion}, we know that $u\in X_{0}^{2s}(\Omega)$. The inequality is simply defined by the definition of $K_{2}$. 
\end{proof}

\section{\label{sec:proof-Comparison}Proof of Theorem~\ref{thm:Comparison-var}}
Before we prove the comparison principle for variational inequality, we have the following lemma, which is based on \cite[Section II-6]{KS80}. 
\begin{lem}
	\label{lem:KS80-lemma1}If $u\in H_{0}^{s}(\Omega)$ is the unique
	solution to Lemma~\ref{lem:LS1}, then 
	\[
	u\le w\quad\text{a.e. in }\Omega\quad\text{for all }w\in K_{0}\cap K_{1}.
	\]
\end{lem}

\begin{proof}
	Substituting $v:=\min\{u,w\}\in K_{0}$ in condition (b) of Lemma~\ref{lem:LS1},
	we have 
	\begin{equation}
		\langle Au,[u-w]_{+}\rangle\le\langle f,[u-w]_{+}\rangle,\label{eq:patch-ver11-3}
	\end{equation}
	because $[u-w]_{+}=u-\min\{u,w\}$. Because $w\in K_{1}$ and $[u-w]_{+}\ge0$,
	then 
	\begin{equation}
		\langle f,[u-w]_{+}\rangle\le\langle Aw,[u-w]_{+}\rangle.\label{eq:patch-ver11-4}
	\end{equation}
	Combining \eqref{eq:patch-ver11-3} and \eqref{eq:patch-ver11-4},
	we have 
	\[
	\langle Au,[u-w]_{+}\rangle\le\langle Aw,[u-w]_{+}\rangle.
	\]
	Conversely, by using Lemma~\ref{lem:MN17}, we obtain 
	\begin{align*}
		& \langle A[u-w]_{+},[u-w]_{+}\rangle\\
		& =\langle A(u-w),[u-w]_{+}\rangle+\langle A[u-w]_{-},[u-w]_{+}\rangle\\
		& \le\langle A(u-w),[u-w]_{+}\rangle\\
		& \le0.
	\end{align*}
	Hence, we have $\langle A[u-w]_{+},[u-w]_{+}\rangle=0$, which implies
	\[
	[u-w]_{+}=0\text{ a.e. in }\Omega,
	\]
	which is our desired result. 
\end{proof}
Now, we are ready to prove Theorem~\ref{thm:Comparison-var}. 
\begin{proof}
	[Proof of Theorem~{\rm \ref{thm:Comparison-var}}] Because $f_{1}\le f_{2}$
	in $H^{-s}(\Omega)$, then $Au_{2}\ge f_{2}\ge f_{1}$, i.e. $u_{2}\in K_{1}^{1}$.
	Moreover, because $\psi_{1}\le\psi_{2}$ a.e. in $\Omega$, we have
	$u_{2}\ge\psi_{2}\ge\psi_{1}$, i.e. $u_{2}\in K_{0}^{1}$. This shows
	that 
	\[
	u_{2}\in K_{0}^{1}\cap K_{1}^{1}.
	\]
	Therefore, we can apply Lemma~\ref{lem:KS80-lemma1} with $(f,\psi)=(f_{1},\psi_{1})$,
	that is, $K_{0}=K_{0}^{1}$ and $K_{1}=K_{1}^{1}$), as well as $u=u_{1}$
	and $w=u_{2}$. Hence we obtain $u_{1}\le u_{2}$ a.e. in $\Omega$. 
\end{proof}

\section{\label{sec:Proof-of-Uniqueness}Proof of Theorem~\ref{thm:Uniqueness}}

Employing the ideas in \cite{AK19} (see also \cite{Bre11}), we can obtain the following chain-rule for the function $t\mapsto\phi(u(t))$. Indeed, by choosing $V = X_{0}^{2s}(\Omega)$ and $H = L^{2}(\Omega)$ in Theorem~\ref{thm:chain-rule}, together with \eqref{eq:duality-pair}, we can obtain the following lemma. 
\begin{lem}
\label{lem:chain}If $u\in H^{1}(0,T;L^{2}(\Omega))\cap L^{2}(0,T;X_{0}^{2s}(\Omega))$, then  
\begin{enumerate}
\item [(i)] the function 
\[
t\mapsto\phi(u(t))=\frac{1}{2}\int_{\Omega}|(-\Delta)^{s/2}u(x,t)|^{2}\,dx
\]
belongs to $W^{1,1}(0,T)$; 
\item [(ii)] for a.e. $t\in(0,T)$, it holds that 
\[
\frac{d}{dt}\bigg[\int_{\Omega}|(-\Delta)^{s/2}u(x,t)|^{2}\,dx\bigg]=2\int_{\Omega}(\partial_{t}u)((-\Delta)^{s}u)\,dx.
\]
\item [(iii)] $u\in\mathcal{C}([0,T];H_{0}^{s}(\Omega))$. 
\end{enumerate}
\end{lem}

Recall the following inequality 
\begin{equation}
|a_{+}-b_{+}|^{2}\le|a_{+}-b_{+}||a-b|=(a_{+}-b_{+})(a-b)\quad\text{for all }a,b\in\mathbb{R}.\label{eq:positive-part-ineq}
\end{equation}
Finally, we are ready to prove Theorem~\ref{thm:Uniqueness}. 
\begin{proof}
[Proof of Theorem~{\rm \ref{thm:Uniqueness}}] Let $u_{1},u_{2}$ be a strong
solution of \eqref{eq:diff1}, \eqref{eq:diff-Dirichlet}, and \eqref{eq:diff-initial},
with initial conditions $u_{1,0},u_{2,0}\in H_{0}^{s}(\Omega)$, respectively. Let
$u=u_{1}-u_{2}$. By Lemma~\ref{lem:chain}, we have 
\begin{align*}
 & \frac{d}{dt}\bigg[\int_{\Omega}|(-\Delta)^{s/2}u(x,t)|^{2}\,dx\bigg]\\
= & 2\int_{\Omega}(\partial_{t}u)((-\Delta)^{s}u)\,dx\\
= & -2\int_{\Omega}\bigg([-(-\Delta)^{s}u_{1}+f_{1}]_{+}-[-(-\Delta)^{s}u_{2}+f_{2}]_{+}\bigg)\times\\
 & \qquad\times\bigg((-(-\Delta)^{s}u_{1}+f_{1})-(-(-\Delta)^{s}u_{2}+f_{2})-f_{1}+f_{2}\bigg)\,dx.
\end{align*}
Using \eqref{eq:positive-part-ineq} with $a=-(-\Delta)^{s}u_{1}+f_{1}$
and $b=-(-\Delta)^{s}u_{2}+f_{2}$, we reach 
\begin{align*}
 & \frac{d}{dt}\bigg[\int_{\Omega}|(-\Delta)^{s/2}u(x,t)|^{2}\,dx\bigg]\\
\le & -2\int_{\Omega}\bigg|[-(-\Delta)^{s}u_{1}+f_{1}]_{+}-[-(-\Delta)^{s}u_{2}+f_{2}]_{+}\bigg|^{2}\,dx\\
 & +2\int_{\Omega}([-(-\Delta)^{s}u_{1}+f_{1}]_{+}-[-(-\Delta)^{s}u_{2}+f_{2}]_{+})(f_{1}-f_{2})\,dx\\
\le & -\int_{\Omega}\bigg|[-(-\Delta)^{s}u_{1}+f_{1}]_{+}-[-(-\Delta)^{s}u_{2}+f_{2}]_{+}\bigg|^{2}\,dx+\int_{\Omega}|f_{1}-f_{2}|^{2}\,dx\\
= & -\int_{\Omega}|\partial_{t}u|^{2}\,dx+\int_{\Omega}|f_{1}-f_{2}|^{2}\,dx
\end{align*}
for a.e. $t\in(0,T)$. This implies that 
\begin{align*}
 & \int_{\Omega}|\partial_{t}u_{1}-\partial_{t}u_{2}|^{2}\,dx\\
 & +\frac{d}{dt}\bigg[\int_{\Omega}|(-\Delta)^{s/2}u_{1}(x,t)-(-\Delta)^{s/2}u_{2}(x,t)|^{2}\,dx\bigg]\\
\le & \int_{\Omega}|f_{1}-f_{2}|^{2}\,dx.
\end{align*}
Hence, we obtain 
\begin{align*}
 & \int_{0}^{T}\|\partial_{t}u_{1}(t)-\partial_{t}u_{2}(t)\|_{L^{2}(\Omega)}^{2}\,dx\\
 & +\sup_{t\in[0,T]}\int_{\Omega}|(-\Delta)^{s/2}u_{1}(x,t)-(-\Delta)^{s/2}u_{2}(x,t)|^{2}\,dx\\
\le & 2\bigg[\|(-\Delta)^{s/2}u_{1,0}-(-\Delta)^{s/2}u_{2,0}\|_{L^{2}(\Omega)}^{2}+\int_{0}^{T}\|f_{1}(t)-f_{2}(t)\|_{L^{2}(\Omega)}^{2}\,dt.
\end{align*}
Using \eqref{eq:norm-equiv}, we obtain \eqref{eq:conti-sol}, which is our desired result.  
\end{proof}

\section{\label{sec:Sketch-Existence}Sketch of the Proof of Theorem~\ref{thm:Existence}}

The proof of Theorem~\ref{thm:Existence} is similar to that given in \cite{AK19}.
Here, we sketch some of the main ideas. Using $\tau$, we denote a division $\{t_{0},t_{1},\cdots,t_{m}\}$
of the interval $[0,T]$ given by 
\[
0=t_{0}<t_{1}<\cdots<t_{m}=T,\quad\tau_{k}:=t_{k}-t_{k-1}\quad\text{for }k=1,\cdots,m.
\]
Now, we shall approximate $u(x,t)$ by time discretization of $u_{k}(x)$.
By Lemma~\ref{lem:solvability}, we can construct $u_{k}\in H_{0}^{s}(\Omega)$
using the implicit Euler scheme 
\begin{equation}
\frac{u_{k}-u_{k-1}}{\tau_{k}}=[-(-\Delta)^{s}u_{k}+f_{k}]_{+}\quad\text{a.e. in }\Omega,\label{eq:Euler1}
\end{equation}
where $f_{k}\in L^{2}(\Omega)$ is given by 
\[
f_{k}(x):=\frac{1}{\tau_{k}}\int_{t_{k-1}}^{t_{k}}f(x,r)\,dr.
\]
For a given $u_{0}\in H_{0}^{s}(\Omega)$, we shall inductively define
$u_{k}\in H_{0}^{s}(\Omega)$ for $k=1,2,\cdots,m$ as a minimizer
of the functional 
\begin{equation}
J_{k}(v):=\frac{1}{2\tau_{k}}\int_{\Omega}|v|^{2}\,dx+\frac{1}{2}\int_{\Omega}|(-\Delta)^{s}v|^{2}\,dx-\bigg\langle\frac{u_{k-1}}{\tau_{k}}+f_{k},v\bigg\rangle\quad\text{for }v\in H_{0}^{s}(\Omega)\label{eq:Euler2}
\end{equation}
subject to 
\begin{equation}
v\in K_{0}^{k}:=\{v\in H_{0}^{s}(\Omega):v\ge u_{k-1}\text{ a.e. in }\Omega\}.\label{eq:Euler3}
\end{equation}
Similar to that given in \cite[Lemma 4.2]{AK19}, the scheme indeed works: 
\begin{lem}\label{lem:discretization}
Suppose that 
\begin{equation}\label{eq:Euler-assump}
u_{0}\in H_{0}^{s}(\Omega),\quad(-\Delta)^{s}u_{0}\in\mathcal{M}(\overline{\Omega}),\quad[(-\Delta)^{s}u_{0}]_{+}\in L^{2}(\Omega).
\end{equation}
For each $k=1,2,\cdots,m$, there exists a unique element $u_{k}\in K_{0}^{k}:=\{v\in H_{0}^{s}(\Omega):v\ge u_{k-1}\text{ a.e. in }\Omega\}$
that minimize \eqref{eq:Euler2} subject to \eqref{eq:Euler3}. Moreover,
for each $k=1,2,\cdots,m$, the minimizer $u_{k}\in X_{0}^{2s}(\Omega)$
verifies \eqref{eq:Euler1}, that is, 
\begin{align}
u_{k}-u_{k-1} & \ge0\quad\text{a.e. in }\Omega,\nonumber \\
g_{k} & \ge0\quad\text{a.e. in }\Omega,\label{eq:Euler4b} \\
\langle g_{k},u_{k}-u_{k-1}\rangle & =0,\label{eq:Euler4c}
\end{align}
where 
\[
g_{k}:=\frac{u_{k}-u_{k-1}}{\tau_{k}}+(-\Delta)^{s}u_{k}-f_{k}.
\]
Furthermore,  
\begin{align}
\langle g_{k},v-u_{k}\rangle & \ge0\quad\text{for all }v\in K_{0}^{k}, \label{eq:Euler4d}\\
\langle g_{k}+f_{k}-(-\Delta)^{s}u_{k-1},v-u_{k}\rangle & \ge0\quad\text{for all }v\in K_{1}^{k}, \label{eq:Euler4e}
\end{align}
where 
\[
K_{1}^{k}:=\bigg\{ v\in H_{0}^{s}(\Omega):\frac{v-u_{k-1}}{\tau_{k}}+(-\Delta)^{s}v-f_{k}\ge0\text{ in }H^{-s}(\Omega)\bigg\}.
\]
Moreover, it holds that 
\begin{equation}
0\le g_{k}\le[(-\Delta)^{s}u_{k-1}-f_{k}]_{+}\quad\text{a.e. in }\Omega \label{eq:Euler4f}
\end{equation}
for each $k=1,\cdots,m$. 
\end{lem}

\begin{proof}
	\textbf{Step 1: The case when $k=1$.} Set 
	\[
	\sigma:=\frac{1}{\tau_{1}}>0,\quad A:=A_{\sigma}=(-\Delta)^{s}+\frac{1}{\tau_{1}},\quad f=f_{1}+\frac{u_{0}}{\tau_{1}}\in L^{2}(\Omega),\quad\psi:=u_{0}\in H_{0}^{s}(\Omega).
	\]
	Note that $f$ and $\psi$ satisfy \eqref{eq:LS-assump1}. By \eqref{eq:Euler-assump},
	we know that 
	\[
	A\psi-f=\overbrace{(-\Delta)^{s}u_{0}}^{\in\mathcal{M}(\overline{\Omega})}-\overbrace{f_{1}}^{\in L^{2}(\Omega)}.
	\]
	Because $[(-\Delta)^{s}u_{0}]_{+}\in L^{2}(\Omega)$, using Lemma~\ref{lem:A1},
	we know that 
	\[
	[A\psi-f]_{+}\in L^{2}(\Omega).
	\]
	Therefore, $[A\psi]_{+}\in L^{2}(\Omega)$ for all $\lambda\ge0$,
	that is, \eqref{eq:LS-assump2} holds. Therefore, by Lemma~\ref{lem:LS1}, and \ref{lem:LS2}, there exists a unique $u_{1}\in K_{0}^{1}$ of
	$J_{1}$ given in \eqref{eq:Euler2}, and \eqref{eq:Euler4b}--\eqref{eq:Euler4e}
	follow immediately from the fact $u_{1}\in K_{0}^{1}$ and Lemma~\ref{lem:LS1}(b)(c)(d).
	By Theorem~\ref{thm:LS-ineq}, we know that $u_{1}\in X_{0}^{2s}(\Omega)$. 
	
	\textbf{Step 2: The case when $k=2,\cdots,m$.} Set 
	\[
	\sigma:=\frac{1}{\tau_{k}}>0,\quad A:=A_{\sigma}=(-\Delta)^{s}+\frac{1}{\tau_{k}},\quad f=f_{k}+\frac{u_{k-1}}{\tau_{k}}\in L^{2}(\Omega),\quad\psi:=u_{k-1}\in H_{0}^{s}(\Omega).
	\]
	Repeating the steps above, we can obtain analogous results, and show that 
	\[
	u_{k}\in X_{0}^{2s}(\Omega)\quad\text{for all }k=2,\cdots,m.
	\]
	
	\textbf{Step 3: Conclusion.} Finally, by Theorem~\ref{thm:LS-ineq},
	we have 
	\begin{equation}
		f_{k}+\frac{u_{k-1}}{\tau_{k}}\le\frac{u_{k}}{\tau_{k}}+(-\Delta)^{s}u_{k}\le\max\bigg\{ f_{k}+\frac{u_{k-1}}{\tau_{k}},\frac{u_{k-1}}{\tau_{k}}+(-\Delta)^{s}u_{k-1}\bigg\}\quad\text{for a.e. }x\in\Omega,\label{eq:patch1a}
	\end{equation}
	which is equivalent to \eqref{eq:Euler4f}. Indeed, the right-hand side of \eqref{eq:patch1a} is in $L^{2}(\Omega)$ can be shown using
	\eqref{eq:Euler-assump} and Lemma~\ref{lem:A1}. 
\end{proof}

With Lemma~\ref{lem:discretization}, we are now ready to sketch the proof of Theorem~\ref{thm:Existence}. 

\begin{proof}
[Sketch of the proof of Theorem~{\rm \ref{thm:Existence}}] We define the
piecewise linear interpolant $u_{\tau}\in W^{1,\infty}(0,T;H_{0}^{s}(\Omega))$
of $\{u_{k}\}$, and the piecewise constant interpolants $\overline{u}_{\tau}\in L^{\infty}(0,T;H_{0}^{s}(\Omega))$
and $\overline{f}_{\tau}\in L^{\infty}(0,T;L^{2}(\Omega))$ of $\{u_{k}\}$
and $\{f_{k}\}$, respectively, by 
\begin{align*}
u_{\tau}(t):=u_{k-1}+\frac{t-t_{k-1}}{\tau}(u_{k}-u_{k-1}) & \quad\text{for }t\in[t_{k-1},t_{k}]\text{ and }k=1,\cdots,m,\\
\overline{u}_{\tau}(t):=u_{k},\quad\overline{f}_{\tau}(t):=f_{k} & \quad\text{for }t\in(t_{k-1},t_{k}]\text{ and }k=1,\cdots,m.
\end{align*}
Summing \eqref{eq:Euler4c} for $k=1,\cdots,\ell$ for arbitrary $\ell\le m$,
we obtain 
\begin{equation}
\int_{0}^{t}\|\partial_{t}u_{\tau}(r)\|_{L^{2}(\Omega)}^{2}\,dr+\|\overline{u}_{\tau}(t)\|_{H_{0}^{s}(\Omega)}^{2}\le\|u_{0}\|_{H_{0}^{s}(\Omega)}^{2}+\int_{0}^{t}\|\overline{f}_{\tau}(r)\|_{L^{2}(\Omega)}^{2}\,dr \label{eq:patch2}
\end{equation}
for all $t\in[0,T]$. Hence, we obtain a uniform bound: 
\begin{align}
 & \|\partial_{t}u_{\tau}\|_{L^{2}(0,T;L^{2}(\Omega))}^{2}+\|\overline{u}_{\tau}\|_{L^{\infty}(0,T;H_{0}^{s}(\Omega))}^{2}+\|u_{\tau}\|_{L^{\infty}(0,T;H_{0}^{s}(\Omega))}^{2} \nonumber \\
\le & C\bigg[\|u_{0}\|_{H_{0}^{s}(\Omega)}^{2}+\|f\|_{L^{2}(0,T;L^{2}(\Omega))}\bigg]. \label{eq:uniform-bound}
\end{align}
Note that 
\begin{equation}
\overline{f}_{\tau} \rightarrow f \quad \text{in }L^{2}(0,T;L^{2}(\Omega))\text{ strong} \label{eq:conv-inter}
\end{equation}
as $m \rightarrow \infty$, such that $|\tau| \rightarrow 0$. Indeed, $\{ \overline{f}_{\tau} \}$ is bounded in $L^{2}(0,T;L^{2}(\Omega))$, and 
\[
\| \overline{f}_{\tau} \|_{L^{2}(0,T;L^{2}(\Omega))} \le \| f \|_{L^{2}(0,T;L^{2}(\Omega))}.
\] 

Using \eqref{eq:uniform-bound}, we can show that up to a (non-relabeled) sub-sequence, there exists 
\[
u \in L^{\infty}(0,T;H_{0}^{s}(\Omega)) \cap H^{1}(0,T;L^{2}(\Omega)) \cap \mathcal{C}([0,T];L^{2}(\Omega))
\]
such that  
\begin{subequations}
\begin{align}
\overline{u}_{\tau}\rightarrow u & \quad\text{in }L^{\infty}(0,T;H_{0}^{s}(\Omega))\text{ weak-}* \label{eq:conv-a}\\
u_{\tau}\rightarrow u & \quad\text{in }L^{\infty}(0,T;H_{0}^{s}(\Omega))\text{ weak-}* \label{eq:conv-b}\\
 & \quad\text{in }H^{1}(0,T;L^{2}(\Omega))\text{ weak} \label{eq:conv-c}\\
 & \quad\text{in }\mathcal{C}([0,T];L^{2}(\Omega))\text{ strong} \label{eq:conv-d}\\
u_{\tau}(T)\rightarrow u(T) & \quad\text{in }H_{0}^{s}(\Omega)\text{ weak}. \label{eq:conv-e}
\end{align}
\end{subequations}
Note that the convergence of \eqref{eq:conv-a}, \eqref{eq:conv-b}, and \eqref{eq:conv-c} is followed by \eqref{eq:uniform-bound} and the Banach-Alaoglu Theorem. Moreover, the limits of $u_{\tau}$ and $\overline{u}_{\tau}$ are identical, which can be seen using the following inequality: 
\begin{align}
\| u_{\tau}(t) - \overline{u}_{\tau}(t) \|_{L^{2}(\Omega)} & = \bigg| \frac{t_{k}-t}{\tau_{k}}\bigg| \| u_{k} - \overline{u}_{k} \|_{L^{2}(\Omega)} \nonumber \\
& \le \bigg\| \frac{u_{k} - u_{k-1}}{\tau_{k}} \bigg\|_{L^{2}(\Omega)} \tau_{k} \nonumber \\
& \le C |\tau|^{1/2}
\end{align}
for all $t \in (t_{k-1},t_{k}]$, and for all $k=1,2,\cdots ,m$. Furthermore, using the compact embedding \eqref{eq:cpt-emb}
and Ascoli's compactness lemma, we obtain the strong convergence \eqref{eq:conv-d}. Using \eqref{eq:patch2}, we know that $u_{\tau}(T) = u_{m}$ is bounded in $H_{0}^{s}(\Omega)$. Therefore, we can obtain \eqref{eq:conv-e} from \eqref{eq:conv-d}. 

Next, following the arguments in \cite{AK19}, we can show that $\overline{u}_{\tau}$
is uniformly bounded in $L^{2}(0,T;X_{0}^{2s}(\Omega))$: We first rewrite \eqref{eq:Euler4f} as 
\begin{equation}
-\frac{u_{k}-u_{k-1}}{\tau_{k}} + f_{k} \le (-\Delta)^{s}u_{k} \le \max \bigg\{ -\frac{u_{k}-u_{k-1}}{\tau_{k}} + f_{k} , -\frac{u_{k}-u_{k-1}}{\tau_{k}} + (-\Delta)^{s}u_{k-1} \bigg\} \label{eq:patch3}
\end{equation}
a.e. in $\Omega$ for $k=2,3,\cdots,m$. Because $(u_{k} - u_{k-1})/\tau_{k} \ge 0$ a.e. in $\Omega$ (because $u\in K_{0}^{k}$), from \eqref{eq:exist2}, we have 
\[
\text{RHS of \eqref{eq:patch3}} \le \max \{ f_{k} , (-\Delta)^{s}u_{k-1} \} \le \max \{ f^{*} , (-\Delta)^{s}u_{k-1} \}
\]
a.e. in $\Omega$. Therefore, by repeating the inequality above, we have 
\[
(-\Delta)^{s}u_{k} \le \max \{ f^{*} , (-\Delta)^{s}u_{1} \} \quad \text{for all }k=2,3,\cdots,m. 
\]
By Lemma~\ref{lem:A1} and \eqref{eq:Euler4f}, we have 
\begin{align*}
\max \{ f^{*} , (-\Delta)^{s}u_{1} \} & \le \max \bigg\{ f^{*} , [(-\Delta)^{s}u_{0} - f_{1}]_{+} - \frac{u_{1}-u_{0}}{\tau_{1}} + f_{1} \bigg\} \\
& \le \max \{ f^{*} , [(-\Delta)^{s}u_{0} - f_{1}]_{+} + f_{1} \} \\
& \le \max \{ f^{*} , [(-\Delta)^{s}u_{0}]_{+} + [- f_{1}]_{+} + f_{1} \} \\
& = \max \{ f^{*} , [(-\Delta)^{s}u_{0}]_{+} + [f_{1}]_{+} \} \\
& \le [(-\Delta)^{s}u_{0}]_{+} + [f^{*}]_{+}.
\end{align*}
Hence, we obtain
\[
-\frac{u_{k}-u_{k-1}}{\tau_{k}} + f_{k} \le [(-\Delta)^{s}u_{0}]_{+} + [f^{*}]_{+} \quad \text{a.e. in }\Omega.
\]
Therefore, with \eqref{eq:Euler-assump}, we have 
\[
\| (-\Delta)^{s} u_{k} \|_{L^{2}(\Omega)}^{2} \le 2 \bigg( \| [(-\Delta)^{s}u_{0}]_{+} \|_{L^{2}(\Omega)}^{2} + \|f^{*}\|_{L^{2}(\Omega)}^{2} + \bigg\| \frac{u_{k} - u_{k-1}}{\tau_{k}} \bigg\|_{L^{2}(\Omega)}^{2} + \| f_{k} \|_{L^{2}(\Omega)}^{2} \bigg)
\]
for $k=1,2,\cdots,m$. Using \eqref{eq:uniform-bound} and \eqref{eq:conv-inter}, we obtain 
\begin{align}
& \int_{0}^{T} \| (-\Delta)^{s} \overline{u}_{\tau} (t) \|_{L^{2}(\Omega)}^{2} \,dx \nonumber \\
& \le CT(\|f^{*}\|_{L^{2}(\Omega)}^{2} + 1) + C \int_{0}^{T} \| \partial_{t} u_{\tau}(t) \|_{L^{2}(\Omega)}^{2} \,dt + C \int_{0}^{T} \|\overline{f}_{\tau}(t)\|_{L^{2}}^{2} \,dt \le C.
\end{align}
Therefore, up to a (non-relabeled) sub-sequence, we have 
\[
(-\Delta)^{s}\overline{u}_{\tau}\rightarrow(-\Delta)^{s}u\quad\text{weakly in }L^{2}(0,T;L^{2}(\Omega)),
\]
which shows that $u(t)\in X_{0}^{2s}(\Omega)$. Therefore, the piecewise
constant interpolant $\overline{g}_{\tau}$ of $\{g_{k}\}$ defined
by 
\[
\overline{g}_{\tau}(t):=g_{k}=\frac{u_{k}-u_{k-1}}{\tau_{k}}+(-\Delta)^{s}u_{k}-f_{k}\quad\text{for }t\in(t_{k-1},t_{k}]
\]
which converges to 
\[
g:=\partial_{t}u+(-\Delta)^{s}u-f\quad\text{weakly in }L^{2}(0,T;L^{2}(\Omega)).
\]
Finally, following Minty's approach in \cite{AK19}, which involves sub-differential calculus, we can show that $u$ solves \eqref{eq:diff1} a.e. in $\Omega\times(0,T)$.  
\end{proof}
\appendix

\section{\label{sec:Aux-Lemma}Some Auxiliary Lemmas}

In this section, we shall reference some results in \cite{AK19}. 
\begin{lem}
\label{lem:A1}Let $1\le p\le\infty$, $\mu\in\mathcal{M}(\overline{\Omega})$
and $\zeta\in L^{p}(\Omega)$. It holds that 
\[
[\mu]_{+}\in L^{p}(\Omega)\quad\text{if and only if}\quad[\mu+\zeta]_{+}\in L^{p}(\Omega).
\]
Moreover, $\|[\mu+\zeta]_{+}\|_{L^{p}(\Omega)}\le\|[\mu]_{+}\|_{L^{p}(\Omega)}+\|[\zeta]_{+}\|_{L^{p}(\Omega)}$.
An analogous conclusion also holds for the negative part. Here, $\mu+\zeta$ means that $\mu+\mu_{\zeta}$. 
\end{lem}

We also need the following lemma, which can be found in \cite[Lemma 2.6]{War15}.
Here, we give a simple alternative proof.
\begin{lem}
\label{lem:A3}Let $s\in(0,1)$. If $v\in H_{0}^{s}(\Omega)$,
then $[v]_{+}\in H_{0}^{s}(\Omega)$ and $\|[v]_{+}\|_{H_{0}^{s}(\Omega)}\le\|v\|_{H_{0}^{s}(\Omega)}$. 
\end{lem}
\begin{proof}
Note that $|a_+|\le |a|$ and $|a_+-b_+|\le |a-b|$ for $a,b\in\R$.
Then, the inequality $\|[v]_{+}\|_{H_{0}^{s}(\Omega)}\le\|v\|_{H_{0}^{s}(\Omega)}$
immediately follows from 
\begin{align*}
 &\|v_+\|_{L^2(\Omega)}^2 =\int_\Omega |(v(x))_+|^2\,dx
  \le \int_\Omega |v(x)|^2\,dx =\|v\|_{L^2(\Omega)}^2,\\
  &[v_+]_{H_{0}^{s}(\Omega)}^{2}=
  \int_{\Omega}\int_{\Omega}\frac{((v(x))_+-(v(z))_+)^{2}}{|x-z|^{n+2s}}\,dx\,dz.
  \le  \int_{\Omega}\int_{\Omega}\frac{(v(x)-v(z))^{2}}{|x-z|^{n+2s}}\,dx\,dz
  =[v]_{H_{0}^{s}(\Omega)}^{2}.
\end{align*}
For the case $s=1/2$, 
\eqref{eq:Lions-Magnes-condition} is additionally required. It is also shown as
\[
\int_{\Omega}\frac{v_+(x)^{2}}{{\rm dist}(x,\partial\Omega)}\,dx
\le
\int_{\Omega}\frac{v(x)^{2}}{{\rm dist}(x,\partial\Omega)}\,dx<\infty.
\]
\end{proof}
\begin{lem}\label{densitylemma}
  We suppose that $v\in H_0^s(\Omega)$ and $v(x)\ge 0$ a.e. $x\in\Omega$.
  Then, there exists a sequence $\{ v_{m}\}_{m=1}^\infty \subset \mathcal{C}_{c}^{\infty}(\Omega)$ such that
  $v_{m}(x)\ge 0$ for $x\in\Omega$ and $v_{m}\to v$ in $H_0^s(\Omega)$ as $m\to \infty$.
\end{lem}

Indeed, Lemma~\ref{densitylemma} can be proved using exactly the same idea as in \cite[Theorem 3.29]{McL00}. Here, we provide an alternative proof, which relies on Mazur's lemma; see, e.g., \cite[Theorem V.1.2]{Yo80}. 
\begin{lem}
\label{lem:Mazur}Let $(X,\|\bullet\|)$ be a norm space, and $\{x_{k}\}\subset X$. If $x_{\infty}\in X$ exists such that 
\[
x_{k}\text{ converges to }x_{\infty}\text{ weakly,}
\]
then, given any $\epsilon>0$, there exists $\alpha_{1},\cdots,\alpha_{\ell}\ge0$ with $\sum_{k=1}^{\ell}\alpha_{k}=1$ such that 
\[
\bigg\| x_{\infty}-\sum_{k=1}^{\ell}\alpha_{k}x_{k}\bigg\|<\epsilon.
\]
\end{lem}

\begin{proof}[Proof of Lemma~{\rm \ref{densitylemma}}] 
Let $v\in H_{0}^{s}(\Omega)$ with $v\ge0$. By the definition of $H_{0}^{s}(\Omega)$, there exists a sequence $\{\varphi_{k}\}\subset\mathcal{C}_{c}^{\infty}(\Omega)$ such that 
\begin{equation}
\lim_{k\rightarrow\infty}\|v-\varphi_{k}\|_{H_{0}^{s}(\Omega)}=0.\label{eq:densitylemma1}
\end{equation}
Note that $[\varphi_{k}]_{+}\in\mathcal{C}_{c}^{0}(\Omega)$ for each $k$ (but is not necessarily smooth). Observe that 
\begin{align}
 \|v-[\varphi_{k}]_{+}\|_{L^{2}(\Omega)}^{2}
 =\int_{\Omega}|v(x)_+-(\varphi_{k}(x))_+|^{2}\,dx
 \le \int_{\Omega}|v(x)-\varphi_{k}(x)|^{2}\,dx
 =\|v-\varphi_{k}\|_{L^{2}(\Omega)}^{2}.\label{eq:densitylemma2}
\end{align}
Combining \eqref{eq:densitylemma1} and \eqref{eq:densitylemma2}, we know that 
\begin{equation}
\lim_{k\rightarrow\infty}\|v-[\varphi_{k}]_{+}\|_{L^{2}(\Omega)}=0.\label{eq:densitylemma3}
\end{equation}
From Lemma~\ref{lem:A3}, we know that 
\[
\sup_{k}\|[\varphi_{k}]_{+}\|_{H_{0}^{s}(\Omega)}\le\sup_{k}\|\varphi_{k}\|_{H_{0}^{s}(\Omega)}<\infty.
\]
Therefore, there exists a subsequence of $\{[\varphi_{k}]_{+}\}$, here, we still denote by $\{[\varphi_{k}]_{+}\}$, such that 
\[
[\varphi_{k}]_{+}\text{ converges to }\tilde{v}\text{ weakly in }H_{0}^{s}(\Omega)
\]
for some $\tilde{v}\in H_{0}^{s}(\Omega)$. From \eqref{eq:densitylemma3}, we know that $\tilde{v}=v$. 

Applying Lemma~\ref{lem:Mazur}, given any $m\in\mathbb{N}$, there exists $\alpha_{1},\cdots,\alpha_{\ell}\ge0$ with $\sum_{k=1}^{\ell}\alpha_{k}=1$ such that 
\[
\bigg\| v-\sum_{k=1}^{\ell}\alpha_{k} [\varphi_{k}]_{+} \bigg\|_{H_{0}^{s}(\Omega)}<\frac{1}{m}.
\]
Letting $\tilde{v}_{m}:=\sum_{k=1}^{\ell}\alpha_{k} [\varphi_{k}]_{+}$, we have 
\begin{equation}
\tilde{v}_{m}\ge0,\quad\tilde{v}_{m}\in\mathcal{C}_{c}^{0}(\Omega),\quad\text{and}\quad\|v-\tilde{v}_{m}\|_{H_{0}^{s}(\Omega)}<\frac{1}{m}\quad\text{for all }m.\label{eq:densitylemma4}
\end{equation}
For each fixed $m$, using a standard mollifier, we can easily construct a function $v_{m}\in\mathcal{C}_{c}^{\infty}(\Omega)$ with $v_{m}\ge0$ such that 
\begin{equation}
\|v_{m}-\tilde{v}_{m}\|_{H_{0}^{s}(\Omega)}<\frac{1}{m}.\label{eq:densitylemma5}
\end{equation}
Combining \eqref{eq:densitylemma4} and \eqref{eq:densitylemma5}, we conclude that 
\[
\|v-v_{m}\|_{H_{0}^{s}(\Omega)}\le\|v-\tilde{v}_{m}\|_{H_{0}^{s}(\Omega)}+\|v_{m}-\tilde{v}_{m}\|_{H_{0}^{s}(\Omega)}<\frac{2}{m}.
\]
Hence, we obtain our desired lemma. 
\end{proof}

\section{An elementary proof of a chain rule}\label{sec:chain-rule}
Let $H$ and $V$ be real Hilbert spaces such that $V$ is densely and continuously embedded in $H$. Then, $H$ can be continuously embedded in $V'$ using the following identification:
\begin{align*}
	~_{V'}\langle f,v\rangle_V = (f,v)_H\qquad (f\in H,~v\in V).
\end{align*} 
Let $a(\bullet,\bullet)$ be a coercive bounded symmetric bilinear form
on $V\times V$. Then, from the Lax-Milgram theorem \cite{Cia02},
there exists $A\in {\rm Isom}\,(V,V')$ such that
\begin{align*}
	a(u,v)=~_{V'}\langle Au,v\rangle_V\qquad (u,v\in V).
\end{align*}

For a bilinear form $a(\bullet,\bullet)$, we abbreviate $a(u,u)$ to $a(u)$
in this note. The following theorem can be proved by modifying the ideas in \cite[Prop.~1.2 in Chap.~III.1]{Sho97}.
\begin{thm}\label{thm:chain-rule}
	Under the above conditions,
	we suppose that $u\in H^1(0,T;H)$ and that
	$u(t)\in V$ a.e. $t\in (0,T)$ and $Au\in L^2(0,T;H)$.
	Then, $u\in \mathcal{C}^0([0,T];V)$,  
	\begin{align}\label{w11}
		a(u(\bullet)) \in W^{1,1}(0,T),
	\end{align}
	and 
	\begin{align}\label{crule}
		\frac{d}{dt} a(u(t))
		=2(Au(t),u'(t))_H\qquad \mbox{a.e.}~t\in (0,T).
	\end{align}
\end{thm}
\begin{proof}
Because $u\in H^1(0,T;H)$, we can choose a representative $u(t) \in H$ such that $u\in \mathcal{C}^0([0,T];H)$.
We define an extension of $u$ by
\begin{align*}
	\tilde{u}(t):=\left\{
	\begin{array}{ll}
		u(-t)&\text{for } t\in [-T,0),\\
		u(t)&\text{for } t\in [0,T],\\
		u(2T-t)&\text{for } t\in (T,2T].
	\end{array}\right.
\end{align*}
Applying the Friedrichs mollifier $u_\vep:=\eta_\vep *\tilde{u}$,
we obtain $u_\vep\in \mathcal{C}^\infty([0,T];H)$ and 
\begin{align*}
	\left\{
	\begin{array}{ll}  
		u_\vep\rightarrow u \quad&\mbox{in}~\mathcal{C}^0([0,T];H)\text{ strong},\\
		u_\vep '\rightarrow u' \quad&\mbox{in}~L^2([0,T];H)\text{ strong},\\
		Au_\vep \rightarrow Au \quad&\mbox{in}~L^2([0,T];H)\text{ strong},
	\end{array}\right.\qquad
	\mbox{as}~\vep\to 0.
\end{align*}
We remark that $u_\vep\in \mathcal{C}^\infty([0,T],V)$
is implied by $Au_\vep\in \mathcal{C}^\infty([0,T],H)$ and
$A\in \Isom (V,V')$.

Then, given any $\delta >0$ and $\vep >0$, we have
\begin{equation}
	\frac{d}{dt}a(u_\delta(t),u_\vep(t)) = a(u_\delta'(t),u_\vep(t))+a(u_\delta(t),u_\vep'(t)) \label{eq:chain2}
\end{equation}
for all $t \in [0,T]$. Choosing any $t_{0},t_{1} \in [0,T]$, and integrating \eqref{eq:chain2} over the interval $(t_{0},t_{1})$, we obtain
\begin{align}
	& a(u_\delta(t_1),u_\vep(t_1))-  a(u_\delta(t_0),u_\vep(t_0))\notag\\
	& = \int_{t_0}^{t_1} \{a(u_\delta'(t),u_\vep(t))+a(u_\delta(t),u_\vep'(t))\}
	\,dt\notag\\
	& = \int_{t_0}^{t_1} (Au_\vep(t),u_\delta'(t))_H\,dt
	+\int_{t_0}^{t_1} (Au_\delta(t),u_\vep'(t))_H\,dt.\label{ade}
\end{align}
We define  
\begin{align*}
	\mathcal{N}:=\{ t\in [0,T]:u(t)\notin V\}.
\end{align*}
Because $u(t)\in V$ a.e. $t\in (0,T)$, $\mathcal{N}$ is a measure zero set. Given any $\epsilon>0$, we have
\begin{align*}
	a(u_\delta(t_i),u_\vep(t_i))
	=(Au_\vep(t_i),u_\delta(t_i))_H
	\underset{\delta\to 0}{\longrightarrow}
	(Au_\vep(t_i),u(t_i))_H = a(u_\vep(t_i),u(t_i))
\end{align*}
for $i=0,1$. Choosing $t_{0},t_{1}\in [0,T]\setminus \mathcal{N}$ and fixing $\varepsilon>0$, 
taking $\delta\to 0$ in \eqref{ade}, we obtain
\begin{align}
	& a(u(t_1),u_\vep(t_1))-  a(u(t_0),u_\vep(t_0))\notag\\
	& = \int_{t_0}^{t_1} (Au_\vep(t),u'(t))_H\,dt
	+\int_{t_0}^{t_1} (Au(t),u_\vep'(t))_H\,dt\label{ae}
\end{align}
for all $t_{0},t_{1} \in [0,T]\setminus \mathcal{N}$. Because
\begin{align*}
	a(u_\vep(t_i),u(t_i))
	=~_{V'}\langle Au(t_i),u_\vep(t_i)\rangle_V
	\underset{\vep\to 0}{\longrightarrow}
	~_{V'}\langle Au(t_i),u(t_i)\rangle_V
	= a(u(t_i)),
\end{align*}
for $i=0,1$, by taking $\vep\to 0$ in \eqref{ae}, we obtain
\begin{equation}
	a(u(t_1))-  a(u(t_0)) = 2\int_{t_0}^{t_1} (Au(t),u'(t))_H\,dt\label{auu}
\end{equation}
for all $t_{0},t_{1} \in [0,T] \setminus \mathcal{N}$. The assertions
\eqref{w11}  and \eqref{crule} 
follow from \eqref{auu}.

Finally, we want to prove the remaining assertion $u\in \mathcal{C}^0([0,T];V)$.
From \eqref{w11}, we know that $a(u(\bullet))\in L^\infty(0,T)$ holds.
Thus, $u\in L^\infty(0,T;V)$ follows from the coercivity of $a(\bullet,\bullet)$
on $V$, and we can define $M:=\| u\|_{L^\infty(0,T;V)}$.

For each $t\in [0,T]$, there exists a sequence
$\{t_n\}_n\subset [0,T]$ 
such that the following conditions are satisfied:
\begin{align*}
	\lim_{n\to\infty}t_n=t,
	\qquad
	u(t_n)\in V,
	\qquad
	\| u(t_n)\|_V\le M.
\end{align*}
Then, there exists a sub-sequence (without relabeling) of $\{t_n\}_n$
and $\bar{u}(t)\in V$ such that
\begin{align*}
	u(t_n)\rightarrow \bar{u}(t)\quad\text{in }V\text{ weak}\quad(\text{as}~n\to\infty),
	\quad\text{and}\quad
	\| \bar{u}(t)\|_V\le M.
\end{align*}
Because $u(t_n)\to u(t)$ in $H$ strong as $n\to\infty$, then 
$u(t)=\bar{u}(t)$ follows. Hence, we obtain
\begin{align*}
	u(t)\in V\quad \mbox{and}\quad  \| u(t)\|_V\le M\quad
	\mbox{for all}~t\in [0,T].
\end{align*}
In particular, $\mathcal{N}=\emptyset$ holds and $a(u(\bullet))\in \mathcal{C}^0([0,T])$
follows from \eqref{auu}.

We again choose arbitrary $\{t_n\}_n\in [0,T]$
with $t_n\to t\in [0,T]$ as $n\to\infty$.
Repeating the above argument, we can prove that 
\begin{align}\label{utn}
	u(t_n)\rightarrow u(t)\quad\text{in }V\text{ weak}\quad(\text{as}~n\to\infty).
\end{align}
Conversely, $a(u(t_n))\to a(u(t))$ holds from
$a(u(\bullet))\in \mathcal{C}^0([0,T])$.
We define $\| u\|_a:=\sqrt{a(u)}$.
Because $\|\bullet\|_a$ becomes an equivalent norm to $\|\bullet\|_V$
from the coercivity of $a(\bullet,\bullet)$,
$\|u(t_n)\|_a\to \|u(t)\|_a$ as $n\to\infty$
and \eqref{utn} implies that $u(t_n)\to u(t)$ in $V$ strong as
$n\to \infty$.
Hence $u\in \mathcal{C}^0([0,T];V)$ holds.
\end{proof}

We remark that the assertions \eqref{w11} and \eqref{crule} are a special case of \cite[Lemma~3.3]{Bre73}.

\section*{Acknowledgments}

This work was partially supported by JSPS KAKENHI JP20H01812, JP20H00117, JP20KK0058, and MOST 108-2115-M-002-002-MY3.

\end{document}